\documentclass[11pt,letterpaper]{article}
\usepackage[margin=1in]{geometry}
\usepackage[T1]{fontenc}
\usepackage{lmodern,microtype}
\usepackage{amsmath,amssymb,amsthm,mathtools}
\usepackage{booktabs,array,tabularx,enumitem,needspace}
\usepackage[
  backend=biber,
  style=alphabetic,
  maxbibnames=99
]{biblatex}
\usepackage{xcolor,hyperref}
\usepackage[nameinlink,noabbrev,capitalise]{cleveref}
\hypersetup{colorlinks=true,
linkcolor=cyan!0!blue,
citecolor=green!85!black,
urlcolor=blue!40!black,
pdftitle={Subspace Uncertainty and Sharp Sampling Thresholds on the Boolean Cube},pdfauthor={}}
\allowdisplaybreaks[2]
\usepackage[most]{tcolorbox}

\newtcolorbox{questionbox}{
  colback=gray!28,
  colframe=gray!28,
  boxrule=0pt,
  arc=0mm,
  left=8pt,
  right=8pt,
  top=6pt,
  bottom=6pt,
  breakable
}

\newcommand{\Mhard}{M_{d,k}}
\newtheorem{theorem}{Theorem}[section]
\newtheorem{lemma}[theorem]{Lemma}
\newtheorem{proposition}[theorem]{Proposition}
\newtheorem{corollary}[theorem]{Corollary}
\newtheorem{definition}[theorem]{Definition}
\newtheorem{remark}[theorem]{Remark}
\newcommand{\E}{\mathbb E}
\newcommand{\Pp}{\mathbb P}
\newcommand{\R}{\mathbb R}

\newcommand{\1}{\mathbf1}
\newcommand{\cP}{\mathcal P}

\newcommand{\cH}{\mathcal H}
\newcommand{\HH}{\mathsf H_2}
\newcommand{\hSigma}{\widehat\Sigma}
\newcommand{\Nstab}{N_{\mathrm{stab}}}
\newcommand{\Npar}{N_{\mathrm{par}}}

\newcommand{\ip}[2]{\langle#1,#2\rangle}
\newcommand{\Tr}{\operatorname{tr}}

\newcommand{\eps}{\varepsilon}
\DeclareMathOperator{\artanh}{arctanh}
\renewenvironment{abstract}{\begin{center}\bfseries Abstract\end{center}\normalfont\normalsize}{\par}
\title{Subspace Uncertainty and Sharp Sampling Thresholds\\on the Boolean Cube}
\author{Thomas Weinberger\\[0.4em]{\normalsize School of Computer and Communication Sciences, EPFL, Switzerland}\\{\normalsize\texttt{thomas.weinberger@epfl.ch}}}
\date{\today}
\begin{document}
\hypersetup{pageanchor=false}
\pagenumbering{gobble}
\maketitle
\vspace{-2em}
\begin{abstract}

For a known finite-dimensional model subspace, classical Gaussian regression theory guarantees that least squares eventually attains the parametric rate, with prediction error scaling as the subspace dimension divided by the sample size. Yet subspaces with the same number of unknown coefficients can require very different sample sizes to reach this regime. For subspaces of polynomials on the Boolean cube, a fundamental question is how restrictions on degree and ambient dimension jointly determine this threshold. Our main result sharply characterizes the worst-subspace minimax sample threshold over a substantial parameter range, revealing an exponential dependence on degree and precisely quantifying the joint effects of degree, ambient dimension, and subspace dimension.

A small modification of our lower-bound construction reveals a strong separation: a single subspace can admit noiseless recovery with optimal sample complexity independent of degree, yet require essentially the worst-case sample size, exponential in degree, to attain the noisy parametric rate.

To establish the sampling lower bound, we sharpen the Polyanskiy–Samorodnitsky uncertainty principle (J. Funct. Anal., ’19) to uniformly optimal precision and extend its concentration guarantee to large subspaces in which every function concentrates most of its energy on one shared Hamming ball. The sampling upper bound follows from the sharp moment bound of Kirshner and Samorodnitsky (IEEE Trans. Inf. Theory, ’21) combined with empirical-process bounds.

\end{abstract}
\par\medskip
\noindent\rule{\linewidth}{0.4pt}\par
{\small
\noindent\textbf{Keywords:} minimax linear regression,
uncertainty principle, sample complexity, Boolean cube,
low-degree polynomials, hypercontractivity.\par
}
\clearpage
\tableofcontents
\clearpage
\pagenumbering{arabic}
\setcounter{page}{1}
\hypersetup{pageanchor=true}
\addtocontents{toc}{\protect\setcounter{tocdepth}{2}}
\section{Introduction}\label{sec:intro}
A fundamental promise of regression theory is that prediction error eventually scales like the number of unknown coefficients divided by the number of observations \cite{HKZ12,Mou22,EHE24}, a scaling known as the \emph{parametric rate}. How long must one wait before this rate becomes valid? We study this question for low-degree functions on the Boolean cube in a known linear subspace, which we call the \emph{model} throughout this work. For a known model and random inputs, the answer depends on how well the observations reveal the model's population geometry. Low dimensionality alone does not ensure sample-efficient estimation since much of the predictive variation may occur on rarely sampled inputs. This issue connects statistical estimation to the stability of random least squares \cite{CDL13,Mou22,EHE24}.

Besides the subspace dimension, which counts unknown coefficients, two parameters constrain the model: the ambient dimension counts input coordinates, and the degree bounds the order of interactions.
Stricter bounds on these parameters can rule out models that are difficult to estimate. In this work, we seek a quantitative theory of how these restrictions, together with the model dimension, determine the sample size needed for accurate prediction.

\begin{questionbox}
\textbf{Statistical question.} Among known subspaces with a prescribed ambient dimension, degree, and model dimension, how many random examples are needed before the usual parametric prediction rate is guaranteed?
\end{questionbox}

A central distinction in regression is between uniquely determining the ground-truth function from noiseless observations and estimating it accurately from noisy observations. Without noise, the ground truth is uniquely determined precisely when no nonzero function in the model vanishes at every sampled input. With noise, distinct targets must produce responses that differ enough relative to the noise level to be reliably distinguished. We therefore focus on differences between possible ground-truth functions, denoted by $h$ throughout this work. The pointwise energy is $h(x)^2$, and its distribution across inputs determines how informative random observations are about these differences. The total energy $\mathbb E[h(X)^2]$ is the squared population $L_2$ distance between the two functions. Under observed inputs $X_1,\ldots,X_n$, however, their empirical separation is $\sum_{i=1}^n h(X_i)^2$. If this quantity is small relative to the noise variance, the data cannot reliably distinguish the targets.
Concentration of the energy $h^2$ on a rarely sampled set can create precisely this discrepancy. Exact Gaussian minimax results formalize this connection through the spectrum of the empirical covariance matrix \cite{Mou22,EHE24}.

To relate these statistical difficulties to the degree, we use an uncertainty principle describing how much energy a low-degree function can concentrate on a small set of inputs. Polyanskiy and Samorodnitsky identify a sharp asymptotic boundary for this tradeoff \cite{PS19}. 
However, for noise variance $\sigma^2$, the parametric rate allows error of order $\sigma^2/n$ per direction and hence $\sigma^2m/n$ in total. Our lower bound therefore requires finer quantitative control and simultaneous concentration in many directions on the same set, leading to the following geometric question.
\begin{questionbox}
\textbf{Geometric question.} How small can the probability of one set of Boolean inputs be while that set still carries most of the energy of every function in a large low-degree subspace?
\end{questionbox}
This question fits into a broader tradition of using uncertainty principles to understand the limits of reconstruction.
Spatial and Fourier concentration have been connected to recovery from incomplete observations \cite{DS89}.
The subspace perspective also has classical precedents in the study of the effective dimension of band-limited signals concentrated on a common time interval \cite{LP62}.
Our result develops this perspective for low-degree Boolean functions, sharpening the concentration bounds around the asymptotic boundary identified in \cite{PS19} and constructing large subspaces whose functions spatially concentrate on one shared set.
This extension is central to our regression lower bound and suggests broader uses wherever reconstruction accuracy depends on how many directions concentrate simultaneously on a common small region of the input space.

\medskip
\par\noindent\textbf{An illustrative example.}
The following simple one-dimensional model illustrates how degree, ambient dimension, and noise can affect estimation. Let \(d\) be odd for simplicity, let \(1\le k\le d\), let \(X\) be uniform on \(\{-1,1\}^d\), and define
\[
Z_d(x)=\frac{x_1+\cdots+x_d}{\sqrt d},
\qquad
h_{d,k}(x)=
\frac{Z_d(x)^k}{\sqrt{\mathbb E[Z_d(X)^{2k}]}},
\qquad
V_{d,k}=\operatorname{span}\{h_{d,k}\}.
\]
The basis function $h_{d,k}$ has degree at most $k$ and unit energy. 

Positive and negative coordinates typically approximately cancel, but taking higher powers emphasizes inputs with a large coordinate imbalance. To see where the energy is located, fix \(k\) and let \(d\) grow. Then, $Z_d$ approaches a standard Gaussian \(G\), and the limiting energy distribution has density proportional to
\(
z^{2k}e^{-z^2/2}.
\) Its maxima occur at \(z=\pm\sqrt{2k}\), far from the typical order-one values of \(G\). Thus, for any fixed \(\rho>0\), sufficiently large \(k\), and then sufficiently large odd \(d\), the set
\(
E=\{x:|Z_d(x)|\ge\sqrt k\}
\) consisting of two Hamming balls satisfies
\[
\mathbb E[h_{d,k}(X)^2\mathbf1_{E^c}]\le\rho,
\qquad
p:=\mathbb P(X\in E)\le2e^{-k/2}.
\]
The two parameters play complementary roles: \(d\) determines how closely the energy distribution follows the Gaussian approximation, while the degree \(k\) determines how far its peaks lie from typical inputs and how rarely those regions are sampled. This description requires \(d\) large enough both to accommodate the energy near \(\pm\sqrt{2k}\) within \([-\sqrt d,\sqrt d]\) and to make the Gaussian approximation accurate there.

Now consider observations $\{(X_i,Y_i)\}_{i=1}^n$ where
\[
Y_i=\theta h_{d,k}(X_i)+\varepsilon_i,
\qquad
\varepsilon_i\overset{\mathrm{iid}}{\sim}N(0,\sigma^2),
\]with noise independent of the inputs. Without noise, one observation determines \(\theta\) exactly as $Y_i/h_{d,k}(X_i)$. Assume \(1\le n\le e^{k/2}/8\). Applying the classical Le Cam two-point method shows (see Appendix~\ref{app:toy-example}) that
\[
\inf_{\widehat\theta}\sup_{\theta\in\mathbb R}
\mathbb P\!\left\{|\widehat\theta-\theta|^2\ge
\frac{\sigma^2}{32n\rho}\right\}
\ge \frac9{32}.
\]
Thus, if \(\rho<1/(32A)\), the parametric error rate \(A\sigma^2/n\) at confidence \(3/4\) cannot be guaranteed throughout this exponential sample range.

While our eventual lower-bound construction is more intricate than the above toy example, it shares some of the same mechanisms. 
Both use radial polynomials, whose values depend only on the coordinate sum, and concentrate energy on sets of small measure formed by Hamming balls. In both, the relationship between $k$ and $d$ matters. In the toy example, we require sufficiently large $d$ relative to $k$ to leave room for the peaks at $\pm\sqrt{2k}$ and the energy around them. In our later construction, the concentration level approaches the largest possible coordinate sum as $k/d$ approaches $1/2$, so we require this ratio to stay bounded away from $1/2$ to obtain sufficiently strong, uniform concentration guarantees.
A key additional ingredient for our lower-bound construction is to combine many orthogonal directions concentrating on the same small set. Then, the respective estimation errors can accumulate and produce the required linear dependence on the model dimension.

\subsection{Preliminaries}\label{sec:model}
Throughout, let $\mu_d$ be the uniform measure on $\{-1,1\}^d$, $\E[f]=\int f(x)\,d\mu_d(x)$ and $\E_n[f]=n^{-1}\sum_i f(X_i)$. We write $\|f\|_2^2=\E[f^2]$ and $\langle f,g\rangle=\E[fg]$ for functions on the hypercube. 

The Walsh characters $\chi_S(x)=\prod_{i\in S}x_i$ with $S\subseteq[d]$ and $|S|\le k$ form an orthonormal basis for the space $\cP_{d,k}$ of degree-at-most-$k$ functions:
\[
 \cP_{d,k}=\operatorname{span}\{\chi_S:|S|\le k\},
 \qquad D_{d,k}=\dim\cP_{d,k}=\sum_{j=0}^{\min(k,d)}\binom dj.
\]

\smallskip
\noindent\textbf{Statistical setup.} 
The learner receives an $m$-dimensional subspace $V\subseteq\cP_{d,k}$ containing an unknown target function $f_\theta$, together with an evaluable orthonormal basis $\Phi_V=(\phi_1,\ldots,\phi_m)^\top$ (we may drop the subscript when $V$ is clear).
The choice of the orthonormal basis is arbitrary and does not affect the estimation problem. There are no computational restrictions.

The random experiment we consider is
\begin{equation}\label{eq:model}
 X_i\stackrel{\mathrm{iid}}{\sim}\mu_d,\qquad
 Y_i=f_\theta(X_i)+\xi_i,\qquad
 f_\theta=\theta^\top\Phi_V,\qquad
 \xi_i\stackrel{\mathrm{iid}}{\sim}N(0,\sigma^2),
\end{equation}
with independent inputs and noise, $\sigma>0$, and unknown $\theta\in\R^m$. Both $V$ and $\theta$ are fixed before sampling.
Given the observations $\{(X_i,Y_i)\}_{i=1}^n$, the learner outputs $\widehat\theta\in\mathbb R^m$, incurring population prediction loss $L(\theta,\widehat\theta)=\E[(f_{\smash{\widehat{\theta}}}-f_\theta)^2]$. By orthonormality, $L(\theta,\widehat\theta)=\|\widehat\theta-\theta\|_2^2$.

The empirical covariance and its population counterpart are
\[
 \hSigma_V=\frac1n\sum_i\Phi_V(X_i)\Phi_V(X_i)^\top,
 \qquad \Sigma_V=\E[\Phi_V(X)\Phi_V(X)^\top]=I_m.
\]

\medskip
\noindent\textbf{Key definitions.} For a nonnegative, possibly infinite random variable $Z$, define
$Q_{1-\delta}(Z)=\inf\{r\ge0:\Pp(Z\le r)\ge1-\delta\}$, with
$\inf\varnothing=+\infty$. We write $Q_{1-\delta,\theta}$
when the probability is taken under the data distribution
with coefficient vector $\theta$. Throughout, we sometimes write $t=\log(1/\delta)$ for the confidence parameter. The supplied-model minimax risk is
\[
 \mathcal R^*_{n,\delta}(V)
 =\inf_{\widehat\theta_V}\sup_{\theta\in\R^m}
 Q_{1-\delta,\theta}\bigl(\|\widehat\theta_V-\theta\|_2^2\bigr).
\]
Probabilities include inputs, noise, and any estimator randomness.

Our main object of interest is the minimax prediction risk for the worst admissible subspace.

\begin{definition}[Worst-subspace minimax risk]
Define
\[
 \mathfrak{R}_{n,\delta}(d,k,m)
 :=\sup_{\substack{V\subseteq\cP_{d,k}\\\dim V=m}}
   \mathcal R^*_{n,\delta}(V)
 =\sup_{\substack{V\subseteq\cP_{d,k}\\\dim V=m}}
   \inf_{\widehat\theta_V}\sup_{\theta\in\R^m}
   Q_{1-\delta,\theta}
   \bigl(\|\widehat\theta_V-\theta\|_2^2\bigr).
\]
\end{definition}

\begin{remark}[Why the subspace is known]
Giving $V$ to the learner separates coefficient estimation from discovering the relevant low-degree directions. 
Knowing $V$ fixes the number of unknown coefficients at $m$, allowing us to isolate how a degree bound $k$ restricts the complexity of the known model and the resulting difficulty of estimation.
Estimation over all of $\mathcal P_{d,k}$ would not isolate the effect of degree at a fixed model dimension, because changing $k$ also changes the number of unknown coefficients.
\end{remark}

We ask how many samples are needed before this risk reaches the parametric rate. The following threshold is the central quantity we determine.
\begin{definition}[Parametric threshold]
Let $\sigma^2>0$ denote the noise variance in \eqref{eq:model}. For a fixed benchmark constant $A\ge16$,
let $\Npar(d,k,m,A,\delta)$ be the least integer $N\ge1$ such that
\[
 \mathfrak{R}_{n,\delta}(d,k,m)
 \le \frac{A\sigma^2\big(m+\log(1/\delta)\big)}{n}
 \qquad\text{for every }n\ge N.
\]
\end{definition}

We take $A\ge16$ to simplify the proof constants; the value $16$ has no special significance. Our results have the same leading exponent for every fixed $A\ge16$.

We study stable sampling, which uniformly preserves a fixed fraction of population squared norms and, at a sufficiently strong level, guarantees the parametric regression rate.
\begin{definition}[Stable-sampling threshold]\label[definition]{def:stab}
For fixed $c\in(0,1)$, let $\Nstab(d,k,m,c,\delta)$ be the least integer
$N\ge1$ such that, for every $n\ge N$ and every fixed eligible $V$,
\[
 \Pp\{\forall h\in V\colon\,\E_n[h^2]\ge c\E[h^2]\}\ge1-\delta.
\]
Equivalently, $\Pp\{\hSigma_V\succeq cI_m\}\ge1-\delta$.
\end{definition}

\par\medskip
\noindent\textbf{Entropy and the uncertainty curve.} Throughout, all logarithms are natural. We define the binary entropy and the uncertainty curve by
\[
 \HH(u)=-u\log u-(1-u)\log(1-u),\qquad
 \Psi(q)=\log2-\HH\!\left(\tfrac12-\sqrt{q(1-q)}\right),
\]
with $0\log0=0$, $u\in[0,1]$, and $q\in[0,1/2]$. We define
\[
 \Lambda_{d,k}=d\Psi(k/d),\qquad
 \Mhard=\binom d{\lfloor k^{1/3}\rfloor}.
\]
\noindent\textbf{Further notation.}
For real numbers, $a\vee b=\max\{a,b\}$, $a\wedge b=\min\{a,b\}$, and $(a)_+=\max\{a,0\}$. For symmetric matrices $A$ and $B$, $A\succeq B$ means $a^\top(A-B)a\ge0$ for every $a$ and the Rayleigh quotient is $\mathcal R_B(a)=a^\top Ba/\|a\|_2^2$ for $a\ne0$. We write \(f\lesssim g\) for a pointwise inequality up to a positive constant, and $f\asymp g$ for both inequalities. Finally, $\chi_r^2=\sum_{j=1}^rG_j^2$ for independent $G_j\sim N(0,1)$ has the chi-squared distribution with $r$ degrees of freedom, and $Z\succeq_{\rm st}Y$ means $\Pp(Z>u)\ge\Pp(Y>u)$ for every real $u$.

\subsection{Contributions}\label{sec:results}

\subsubsection{Sharp sample thresholds for regression}
Our first result answers the statistical question posed above by determining when the worst-subspace minimax risk reaches the parametric rate, with the dependence on degree and ambient dimension captured by the exponent $\Lambda_{d,k}$.
\begin{questionbox}
\begin{theorem}[Minimax sample threshold]\label{thm:main}
Fix $q_0\in(0,1/2)$ and $A\ge16$. For $1\le k\le q_0d$, $1\le m\le D_{d,k}$ and $0<\delta\le1/4$, put $t=\log(1/\delta)$.
If $m\le\Mhard \vee t$, then
\begin{equation*}
 \left|\log\frac{\Npar(d,k,m,A,\delta)}{m+t}-\Lambda_{d,k}\right|
 \le C_{q_0,A}k^{1/3}.
\end{equation*}
The upper bound holds for every feasible $m$ and is attained by ordinary least squares.
\end{theorem}
\end{questionbox}

\begin{remark}[Polynomial Range]
In particular, for every fixed $B>0$, the above two-sided sample bound holds for every feasible $m\le d^B$, with constants depending additionally on $B$: for sufficiently large $k$ depending on $B$, $\Mhard\ge d^B$, while for bounded $k$, the estimate follows from $\Npar\gtrsim m+t$ and $\Lambda_{d,k}\le2k$.
\end{remark}

Taking the worst case over all ambient dimensions gives a particularly simple threshold, valid for every model dimension \(m\).
\begin{corollary}[Unrestricted ambient dimension]
\label[corollary]{cor:unrestricted-d}
For integers $k,m\ge1$, fixed $A\ge16$, and $0<\delta\le1/4$, put $t=\log(1/\delta)$ and define
\[
 \Npar^\infty(k,m,A,\delta)
 :=\sup_{\substack{d\ge1\\m\le D_{d,k}}}
       \Npar(d,k,m,A,\delta).
\]
Then
\[
 \left|\log\frac{\Npar^\infty(k,m,A,\delta)}{m+t}-2k\right|
 \le C_A k^{1/3}.
\]
\end{corollary}
However, retaining the dependence on $d$ can lead to exponentially smaller thresholds when $k\ll d\ll k^{5/3}$; see Appendix~\ref{app:modeldimscaling}.

\subsubsection{The exponential cost of noise}
The following corollary shows that there exist models with optimal noiseless recovery but exponential noisy sample requirements.
At fixed confidence, \(O(m)\) noiseless observations identify every target in the known model, whereas attaining the noisy parametric rate requires \(m\exp\{\Lambda_{d,k}-O(k^{1/3})\}\) observations, an exponential increase in \(k\) since \(\Lambda_{d,k}\ge k\log4\).
Thus noiseless sample complexity alone offers no reliable indication of how much data noisy estimation will require.
The models for which we show this separation are slight variants of the hard instances used to prove \Cref{thm:main}.

\begin{corollary}[Noisy--noiseless separation in a fixed model]
\label{cor:noise-cost}
Fix $q_0\in(0,1/2)$ and $A\ge16$.
For all sufficiently large $k\le q_0d$, depending only on
$A,q_0$, and every $1\le m\le d-1$, there exists an
$m$-dimensional $V\subseteq\cP_{d,k}$ such that, for every
$t\ge\log4$:
\begin{itemize}
\item $C(m+t)$ noiseless observations suffice to recover
every target in $V$ exactly, with probability at least
$1-e^{-t}$ over the sampled inputs.
\item Under Gaussian noise of any variance $\sigma^2>0$,
\[
 \mathcal R^*_{n,e^{-t}}(V)
 >A\sigma^2\frac{m+t}{n}
 \quad\text{if}\quad
 1\le n\le(m+t)e^{\Lambda_{d,k}-C_{A,q_0}k^{1/3}}.
\]
Conversely, OLS attains the parametric error bound whenever
$n\ge(m+t)e^{\Lambda_{d,k}+C_{A,q_0}k^{1/3}}$.
\end{itemize}
Here $C$ is a universal constant, and the subspace $V$ can be chosen independently of $t$ and $\sigma$.
\end{corollary}
Appendix~\ref{app:noise-cost} establishes a more general tradeoff, parametrized by an integer \(R\). It allows model dimensions \(m\le\binom dR-\binom d{R-1}\), with noiseless recovery from \(O(2^R(m+t))\) samples and a noisy sample lower bound of \((m+t)\exp\{\Lambda_{d,k}-O(R+k^{1/3})\}\).

\subsubsection{A quantitative subspace uncertainty principle}
The following result answers the geometric question posed above by quantifying energy concentration jointly over a subspace.
\begin{questionbox}
\begin{theorem}[Subspace uncertainty principle]\label{thm:common}
Fix $q_0\in(0,1/2)$, $\rho\in(0,1)$, and integers $1\le k\le q_0d$. For every set $E$, if some nonzero $h\in\cP_{d,k}$ satisfies $\E[h^2\1_{E^c}]\le\rho \E[h^2]$, then
\begin{equation*}
 \mu_d(E)\ge\exp\{-\Lambda_{d,k}-C_{\rho,q_0}k^{1/3}\}.
\end{equation*}
Conversely, there exist $W\subseteq\cP_{d,k}$ of dimension $\Mhard$ and a permutation-invariant set $E$ such that
\begin{equation*}\label{eq:common-main}
 \mu_d(E)\le\exp\{-\Lambda_{d,k}+C_{\rho,q_0}k^{1/3}\},
 \qquad \forall h\in W:~\E[h^2\1_{E^c}]\le\rho \E[h^2].
\end{equation*}
For sufficiently large $k$, $E$ is a Hamming ball.
\end{theorem}
\end{questionbox}
We refer to $\rho$ as the \emph{leakage}, the fraction of energy permitted outside the common set.
Let $p_\rho(d,k,m)$ be the smallest probability of a set $E$ admitting an $m$-dimensional $W\subseteq\cP_{d,k}$ with $\E[h^2\1_{E^c}]\le\rho\E[h^2]$ for every $h\in W$. The next result shows that the uniform remainder cannot be improved to $o(k^{1/3})$ in general.

\begin{proposition}[Optimal order of the uniform remainder]\label[proposition]{prop:airy-sharp}
There are a fixed $\rho_*\in(0,1)$, integers $d_k\ge k^3$, and $C<\infty$ such that, for all sufficiently large $k$ and all
$1\le m\le\binom{d_k}{\lfloor k^{1/3}\rfloor}$,
\begin{equation*}
 2k^{1/3}\le
 \log\frac1{p_{\rho_*}(d_k,k,m)}-\Lambda_{d_k,k}
 \le Ck^{1/3}.
\end{equation*}
\end{proposition}

\subsection{Related work}\label{sec:related}

\noindent\textbf{Minimax regression and low-degree learning.}
Mourtada relates expected minimax least-squares risk to inverse sample covariance \cite{Mou22}. El Hanchi, Maddison, and Erdogdu characterize the minimax risk through an exact quantile formula that underlies our statistical analysis \cite{EHE24}.
Without a supplied subspace, the sample complexity of learning bounded low-degree functions from noiseless evaluations grows logarithmically with ambient dimension, as shown by the upper and lower bounds in \cite{EI22,EIS23}, respectively.
Related work on Reed--Muller erasure recovery asks whether randomly observed values uniquely determine a low-degree polynomial over a finite field \cite{ASW15}, rather than how Gaussian noise affects estimation accuracy.

\par\smallskip
\noindent\textbf{Random least squares and empirical covariance bounds.}
Hsu, Kakade, and Zhang analyze population prediction error under random design \cite{HKZ12}, while Cohen, Davenport, and Leviatan study least-squares stability and approximation accuracy \cite{CDL13}. Both obtain sufficient sample bounds controlled by $\sup_x\|\Phi_V(x)\|_2^2$ in case of an orthonormal basis.
Our sampling upper bound instead exploits low Boolean degree to control the energy retained after clipping, yielding sharper guarantees in regimes where maximum-value bounds are overly conservative.
Yaskov derives lower bounds on the smallest empirical covariance eigenvalue using truncated moments of one-dimensional projections \cite{Yas14}. Oliveira establishes lower covariance bounds under fourth-moment assumptions and applies them to random-design least squares \cite{Oli16}.
Combining truncation with empirical-process bounds is an established approach: Mendelson uses bounded Lipschitz approximations in the small-ball method \cite{Men14}, and clipped squares together with contraction and concentration in \cite{Men21}.
We follow this approach, using Bousquet's inequality \cite{Bou02} to exploit variance bounds for the clipped functions. Our technical contribution to this approach is a sharp clipping scale for low-degree functions on the Boolean cube.

\par\smallskip
\noindent\textbf{Uncertainty, log-Sobolev inequalities, and sharp moments.} 
Log-Sobolev inequalities relate concentration of a function's energy to its Fourier complexity \cite[Chapter~5]{BLM13}. Samorodnitsky's modified log-Sobolev inequality strengthens this relation and connects it to small-set geometry \cite{Sam08}. Polyanskiy and Samorodnitsky develop nonlinear log-Sobolev inequalities and improved hypercontractivity to obtain their uncertainty curve \cite{PS19}. Their theorem identifies the asymptotic concentration boundary at fixed $k/d$ (see Appendix~\ref{app:baselines}), but does not provide the simultaneous concentration in many directions needed for our dimension-dependent regression lower bound.
We construct such subspaces and obtain a uniformly optimal $O(k^{1/3})$ remainder.
Using moments of order approaching two, classical Bonami--Beckner hypercontractivity \cite{Bon70,Bec75} already gives the exponential factor $\exp\{2k+O(k^{1/3})\}$ (see \Cref{prop:classical-upper}), which is sharp for unrestricted ambient dimension.
For our dimension-dependent results, classical hypercontractivity, even when tuned for the optimal moment parameter, is not sharp enough. Instead, to match the sharper finite-dimensional exponent $\Lambda_{d,k}=d\Psi(k/d)$, we optimize the Kirshner--Samorodnitsky moment bound \cite{KS21}.

\par\smallskip
\noindent\textbf{Harmonic and polynomial tools.}
We use Filmus--Mossel's dimension formula and inner-product identities under exchangeable measures \cite{FM19} to extend one radial concentration construction to many orthogonal directions.
Our sharpness argument uses the Hermite kernel's Airy limit \cite{DG07} to construct a Gaussian polynomial with positive energy on a tail interval, before transferring the construction to the cube.

\section{Technical overview}

\subsection{From the empirical covariance spectrum to minimax risk}
Consider the ordinary least-squares (OLS) estimator $\widehat\theta_{\mathrm{OLS}}\in\operatorname*{arg\,min}_{a\in\mathbb R^m}\frac1n\sum_{i=1}^n\bigl(Y_i-a^\top\Phi_V(X_i)\bigr)^2$. In the Gaussian noise model, elementary algebra shows that a full-rank sample yields an estimation residual distributed as
\[
 (\widehat\theta_{\rm OLS}-\theta)\mid\mathbf X
 \sim N\!\left(0,\frac{\sigma^2}{n}\hSigma_V^{-1}\right).
\]
It turns out that studying this quantity fully captures statistical hardness under this data model. Indeed, the following result, a direct consequence of \cite[Theorem~3, Corollary~6, and Lemma~5]{EHE24}, identifies the corresponding quantile as the minimax risk over all estimators.

\begin{theorem}[El Hanchi--Maddison--Erdogdu]\label[theorem]{thm:exact-minimax}
For a fixed known subspace $V$ in \eqref{eq:model}, let $G=(G_1,\ldots,G_m)\sim N(0,I_m)$ be independent of the sampled inputs, and define
\[
 Z_{n,V}=\begin{cases}
 (\sigma^2/n)G^\top\hSigma_V^{-1}G,&\hSigma_V\text{ invertible},\\
 +\infty,&\hSigma_V\text{ singular}.
 \end{cases}
\]
Then $\mathcal R^*_{n,\delta}(V)=Q_{1-\delta}(Z_{n,V})$, and any measurable least-squares minimizer attains this minimax value. Moreover, $\mathcal R^*_{n,\delta}(V)=+\infty$ if and only if $\Pr(\hSigma_V\text{ singular})\ge\delta$.

\end{theorem}
Let $\lambda_1\le\cdots\le\lambda_m$ denote the eigenvalues of $\hSigma_V$. Conditionally on the sampled inputs, if $\widehat\Sigma_V$ is invertible,
\begin{equation}\label{eq:comparevar}
Z_{n,V}
\overset{d}{=}
\frac{\sigma^2}{n}\sum_{j=1}^m\frac{G_j^2}{\lambda_j}.
\end{equation}
Proving minimax sampling bounds therefore amounts to controlling the spectrum of the empirical covariance matrix.
To control these eigenvalues, we express them in terms of the energy of functions in $V$. Write $h=a^\top\Phi_V$. Orthonormality of the basis gives
\[
 \E[\Phi_V(X)\Phi_V(X)^\top]=I_m,
 \qquad \E[h^2]=\|a\|_2^2,
\]
whereas the sampled inputs give $\E_n[h^2]=a^\top\widehat\Sigma_Va$.
Thus the distance of the empirical covariance to the identity matrix measures how well the sample preserves energy in each model direction.

\medskip
\noindent\textbf{Minimax upper sampling bound.} In particular,
\[
 \lambda_1
 =\inf_{h\in V\setminus\{0\}}
       \frac{\E_n[h^2]}{\E[h^2]}~,
 \qquad
 \widehat\Sigma_V\succeq cI_m
 \ \Longleftrightarrow\
 \forall h\in V:~\E_n[h^2]\ge c\E[h^2].
\]
To prove an upper sampling bound, it therefore suffices to establish stable sampling (see \Cref{def:stab}). In particular, once $\Pp(\widehat\Sigma_V\succeq cI_m)$ is sufficiently large, OLS attains the parametric rate, with an error constant depending on $c$.

\medskip

\noindent\textbf{Minimax lower sampling bound.} 
Suppose that every $h \in V$ has at most a fraction $\rho$ of its energy outside a common set $E$.
Formally,
\begin{equation}\label{eq:outside-matrix}
 \forall h\in V:~\E[h^2\1_{E^c}]\le\rho\E[h^2]
 \quad\Longleftrightarrow\quad
 \E[\Phi_V(X)\Phi_V(X)^\top\1_{E^c}]
 \preceq\rho I_m.
\end{equation}
This suggests that it is likely that many empirical eigenvalues are small, precluding the parametric rate according to \Cref{thm:exact-minimax} and \eqref{eq:comparevar}.
This requires, for appropriate parameters $d,k,m$, the existence of a subspace $V$ and a small set $E$ on which all functions $h \in V$ jointly concentrate their energy. This ingredient is supplied by the achievability direction of our subspace uncertainty principle.


\subsection{Stable sampling and energy anticoncentration via clipping}

A unit polynomial can take very large values on a small fraction of inputs. Concentration estimates based on its maximum value or higher-order moments can therefore require too many samples. We instead control the bounded quantity $h^2\wedge B$, choosing $B$ as small as possible while preserving a fixed positive fraction of every polynomial's population energy. For $\eta\in(0,1)$, define
\[
 b_\eta(h)=\inf\{B>0:\E[h^2\wedge B]\ge\eta\},
 \qquad
 T_{d,k}(\eta)=
 \sup_{\substack{h\in\cP_{d,k}\\\E[h^2]=1}}b_\eta(h).
\]
Thus $T_{d,k}(\eta)$ is the smallest clipping level that preserves at least $\eta$ energy for every unit polynomial.

\paragraph{From clipping to a sampling guarantee.}
For $B\ge T_{d,k}(\eta)$ and $V\subseteq\cP_{d,k}$, orthonormality gives
\[
 \lambda_1(\widehat\Sigma_V)
 =\inf_{\substack{h\in V\\\E[h^2]=1}}\E_n[h^2]
 \ge\inf_{\substack{h\in V\\\E[h^2]=1}}\E_n[h^2\wedge B]
 \ge\eta-
 \sup_{\substack{h\in V\\\E[h^2]=1}}
 \left|\E_n[h^2\wedge B]-\E[h^2\wedge B]\right|.
\]
The last step subtracts the largest sampling deviation from the population lower bound. In \Cref{prop:clip-cov}, symmetrization, contraction, and Bousquet's inequality control these deviations, giving $\lambda_1(\widehat\Sigma_V)\ge c$ with probability at least $1-e^{-t}$ once 
\begin{equation}\label{eq:nupper}
n\ge C_{\eta,c}B(m+t),
\end{equation}
for fixed $0<c<\eta$. Thus a smaller clipping level yields a smaller sampling bound, while retaining a fixed energy fraction keeps the empirical eigenvalues bounded away from zero.

\paragraph{From moments to the clipping level.}
For any unit-norm polynomial and $B,s>0$, we use that
\[
 \E[h^2\wedge B] \ge 1-B^{-s}\E[|h|^{2+2s}].
\]
Moments just above two therefore control the energy lost to clipping. Starting from the sharp moment bound of
\cite[Corollary~1.4]{KS21}, \Cref{lem:cubic} gives
\(
 \log\E[|h|^{2+2s}]
 \le s\Lambda_{d,k}+C_0ks^3.
\)
The crucial feature is the absence of a quadratic error. Indeed, the clipping loss is at most $1-\eta$ whenever
\[
 \log B\ge \Lambda_{d,k}+C_0ks^2+
             \frac{\log(1/(1-\eta))}{s}.
\]
Balancing the last two terms gives $s\asymp k^{-1/3}$ and hence
\begin{equation}\label{eq:clipupper}
\log T_{d,k}(\eta)\le \Lambda_{d,k}+O_{\eta,q_0}(k^{1/3}),
\end{equation}
as proved in \Cref{prop:Tupper}.  We may therefore choose $B=\exp\{\Lambda_{d,k}+C_{\eta,q_0}k^{1/3}\}$ in \eqref{eq:nupper}, obtaining the sampling bound
\[
 n\ge C_{\eta,c}(m+t)
       \exp\{\Lambda_{d,k}+C_{\eta,q_0}k^{1/3}\}.
\]
This establishes the upper bound in \Cref{thm:main}.


The clipping upper bound \eqref{eq:clipupper} also shows that no unit polynomial can concentrate most of its energy on an excessively small set.
Suppose $\E[h^2\1_{E^c}]\le\rho$ and choose $\eta=(1+\rho)/2>\rho$. By \Cref{prop:Tupper}, clipping at $B=\exp\{\Lambda_{d,k}+C_{\eta,q_0}k^{1/3}\}$ preserves at least $\eta$ energy. But
\[
 \eta\le\E[h^2\wedge B]
 \le\E[h^2\1_{E^c}]+\E[B\1_{E}]
 \le\rho+B\mu_d(E).
\]
Consequently,
\[
 \mu_d(E)\ge\frac{\eta-\rho}{B}
 \ge\exp\{-\Lambda_{d,k}-C_{\rho,q_0}k^{1/3}\},
\]
after adjusting the constant. Intuitively, clipping prevents arbitrarily large values on $E$ from compensating for its small probability. This proves the impossibility direction of the uncertainty principle in \Cref{thm:common}.

\subsection{Achievability of joint energy concentration in many directions}

We first construct one polynomial concentrated on a small
Hamming ball $E_u=\{S_D\ge u\}$, where $S_D=x_1+\cdots+x_D$.
This choice serves two purposes: we can conveniently control the probability of such a ball via binomial tail bounds, and its permutation symmetry allows us to extend spatial concentration to many directions.

\par\smallskip
\noindent\textbf{One concentrated radial polynomial.}
Since $E_u$ depends only on $S_D$, we seek a \emph{radial} polynomial, meaning one that also depends only on $S_D$. The Krawtchouk polynomials form an orthonormal basis for these functions, with one basis element at each degree. Expanding in this basis lets us enforce the degree and unit-norm constraints through the coefficients.

The key is to place the nonzero coefficients near the highest permitted degree and taper them smoothly to zero. Multiplication by $S_D$ couples neighboring Krawtchouk coefficients. A smooth taper makes the resulting polynomial $p$ an approximate eigenfunction of this multiplication operator:
\[
 \|(S_D-\tau)p\|_2\ \text{is small},
 \qquad \tau=2\sqrt{\ell(D-\ell)}.
\]
Equivalently, the energy distribution $p^2\,d\mu_D$ places $S_D$ near $\tau$. Choosing a cutoff slightly below $\tau$ therefore captures most of the energy, while a binomial tail bound controls the ball's probability.

For fixed $\rho\in(0,1)$ and $q_0<1/2$,
\Cref{lem:window,lem:scalar-witness} make this precise: for sufficiently large $\ell\le q_0D$, they provide a unit polynomial of degree at most $\ell$ and a cutoff $u$ such that
\begin{equation*}
 \E[p^2\1_{\{S_D<u\}}]\le\rho,
 \qquad
 \mu_D(S_D\ge u)
 \le\exp\{-D\Psi(\ell/D)+C_{\rho,q_0}\ell^{1/3}\}.
\end{equation*}

\par\smallskip
\noindent\textbf{From one polynomial to many directions.}
A radial polynomial is constant on each slice $\{S_d=s\}$. To obtain additional directions, we multiply a radial factor by polynomials that vary within these slices:
\[
 W_r=\{f\,p_r(S_d):f\in\cH_{d,r}\},
 \qquad \deg p_r\le k-r,
\]
where $\cH_{d,r}$ is the space of degree-$r$ harmonic polynomials, defined in \Cref{sec:common}. Every resulting function has degree at most $k$.

Harmonic polynomials are useful because all unit functions of a fixed harmonic degree distribute their energy identically across slices. Thus a shared radial factor changes their energy distribution in the same way.
Specifically, \Cref{lem:harmonic-radial} (a specialization of \cite[Theorem~3.10]{FM19}) reduces this calculation to the coordinate sum on a cube of dimension $d-2r$.
Choosing $p_r$ from the one-dimensional construction on that cube makes multiplication by $p_r(S_d)$ preserve norms and transfers concentration to every function in $W_r$. We call this the \emph{harmonic lift}.

Crucially, different harmonic degrees remain orthogonal both on the whole cube and outside any common ball $E$. Hence, for $h=\sum_{r=0}^R h_r$ with $h_r\in W_r$,
\[
 \E[h^2\1_{E^c}]
 =\sum_{r=0}^R\E[h_r^2\1_{E^c}]
 \le\rho\sum_{r=0}^R\E[h_r^2]
 =\rho\E[h^2],
\]
provided the same cutoff ensures $\E[h_r^2\1_{E^c}]\le\rho\E[h_r^2]$ for every $r$ and every $h_r\in W_r$. The concentration bounds therefore combine without increasing the leakage.

\par\smallskip
\noindent\textbf{Balancing dimension and concentration.}
Combining harmonic degrees through $R$ gives $\dim W_{\le R}=\binom dR$ \cite{FM19}, but leaves less degree for the radial factors. In \Cref{prop:harmonic-lift}, we choose a common cutoff that preserves concentration in every harmonic degree while keeping the ball sufficiently small, giving
\[
 \mu_d(E)\le
 \exp\{-\Lambda_{d,k}+C_{\rho,q_0}(R+k^{1/3})\},
 \qquad
 \E[h^2\1_{E^c}]\le\rho\E[h^2]
 \quad(h\in W_{\le R}).
\]
Finally, taking $R=\lfloor k^{1/3}\rfloor$ yields a subspace of dimension $\binom d{\lfloor k^{1/3}\rfloor}$ while keeping the error in the exponent of order $k^{1/3}$.
This proves the concentration construction in \Cref{thm:common}.

\subsection{Sampling lower bounds via small eigenvalues}

Write $p=\mu_d(E)$ for the probability of the concentration set constructed above. When $m\le\binom d{\lfloor k^{1/3}\rfloor}$, choose an $m$-dimensional subspace $V$ of the concentrating space.
By \eqref{eq:outside-matrix}, its shared concentration bounds the population covariance contribution from $E^c$.

To analyze the empirical covariance, we use the decomposition
\[
 \widehat\Sigma_V=\widehat\Sigma_E+\widehat\Sigma_{E^c}~, \qquad\hSigma_A=\frac1n\sum_{i=1}^n
 \Phi(X_i)\Phi(X_i)^\top\1_{\{X_i\in A\}},\quad A\in\{E,E^c\}.
\]
There are two reasons why this matrix has many small eigenvalues when $np$ is small relative to $m$. First, few observations fall inside $E$. Each contributes a matrix of rank at most one, so $\widehat\Sigma_E$ vanishes on a subspace of dimension close to $m$. Second, the population bound gives $\E[\Tr\widehat\Sigma_{E^c}]\le m\rho$.
Markov's inequality therefore bounds the total outside energy by $O(m\rho)$ with constant high probability.
On the large subspace where the contribution from observations inside $E$ vanishes, this leaves average empirical energy $O(\rho)$ per direction. A constant fraction of the eigenvalues must consequently be $O(\rho)$. \Cref{lem:rare-cov} formalizes this argument using the min--max principle.

These eigenvalues force large minimax error. Conditionally on a nonsingular design with
$\lambda_1,\ldots,\lambda_r\le\lambda$,
\[
 Z_{n,V}\overset{d}{=}
 \frac{\sigma^2}{n}\sum_{j=1}^m\frac{G_j^2}{\lambda_j}
 \ge\frac{\sigma^2}{n\lambda}\sum_{j=1}^rG_j^2
 \overset{d}{=}\frac{\sigma^2}{n\lambda}\chi_r^2.
\]
For $r$ proportional to $m$ and $\lambda=O(\rho)$, this yields error of order $\sigma^2m/(n\rho)$ with constant probability. When $\log(1/\delta)<m$, choosing $\rho$ sufficiently small makes this exceed the target parametric rate.

In the case where $\log(1/\delta)\ge m$, already one concentrated direction suffices since we can choose any $m$-dimensional subspace of $\cP_{d,k}$ containing that direction. 
The probability of missing $E$, together with a Gaussian tail in that direction, makes the failure probability exceed $e^{-t}$. \Cref{prop:rare-par} combines both regimes with \Cref{thm:exact-minimax} to rule out the parametric guarantee for $n\le c(m+t)/p$.
Substituting the bound on $p$ gives the lower threshold $(m+t)\exp\{\Lambda_{d,k}-O(k^{1/3})\}$ when $m\le\binom d{\lfloor k^{1/3}\rfloor}$ or $t\ge m$, proving the lower bound in \Cref{thm:main}.

\subsection{Sharpness of the cube-root remainder}

To prove \Cref{prop:airy-sharp}, we construct a subspace whose functions retain a fixed positive fraction of their energy on a common set of probability at most $\exp\{-\Lambda_{d,k}-c k^{1/3}\}$. This rules out replacing the uniform $O(k^{1/3})$ uncertainty remainder by $o(k^{1/3})$.

We first construct a polynomial under the standard Gaussian measure $\gamma$. Let $H_0,\ldots,H_K$ be the orthonormal Hermite polynomials, put $N=K+1$, and choose $y_N=2\sqrt N+3N^{-1/6}$. Define
\[
 q_K(x)=
 \frac{\sum_{j=0}^K H_j(y_N)H_j(x)}
      {\bigl(\sum_{j=0}^K H_j(y_N)^2\bigr)^{1/2}}.
\]
This is a unit-norm polynomial of degree at most $K$. Its coefficients are chosen to maximize its value at $y_N$ among all such unit polynomials. To show that this large value produces substantial energy over an interval, rather than just at one point, we use the Hermite kernel's Airy limit (see \eqref{eq:airy-limit}). The Airy function is a classical function that describes the behavior of these polynomials at points $x=2\sqrt N+sN^{-1/6}$, where $s$ stays in a fixed bounded interval as $N\to\infty$.
The limit gives
\[
 \int_{2\sqrt N+3N^{-1/6}}^{2\sqrt N+4N^{-1/6}}
 q_K(x)^2\,d\gamma(x)\longrightarrow\alpha>0
 \qquad(K\to\infty),
\]
for some constant $\alpha>0$, as shown in \eqref{eq:airy-mass}.
Thus a fixed positive fraction of energy remains in this shrinking interval far into the Gaussian tail.

To obtain a large subspace of polynomials on the cube, set
$R=\lfloor k^{1/3}\rfloor$ and $K=k-R$.
Similarly to our earlier subspace construction on the Boolean cube, we multiply the radial polynomial by harmonic polynomials to form a subspace
\(
 W=\bigoplus_{r=0}^R
 \{f\,q_K(S_d/\sqrt d):f\in\cH_{d,r}\}.
\)

As before, this ensures that every product has degree at most $K+r\le k$.
For fixed $k$, consider the cube polynomials $q_K(S_d/\sqrt d)$. As $d\to\infty$, by the central limit theorem and moment convergence, their total energy and their energy above a cutoff \(a_k\), fixed independently of \(d\), converge to the corresponding Gaussian quantities. Since the Gaussian polynomial has a fixed positive energy fraction above $a_k$, the cube polynomials do as well for sufficiently large $d$.
The radial-weight identity (\Cref{lem:harmonic-radial}) then gives $\dim W=\binom dR$ and transfers concentration to every function in $W$, including across different harmonic degrees.

The Gaussian interval above lies beyond $a_k=2\sqrt k+k^{-1/6}$ for sufficiently large $k$. Consequently, the shared concentration set can be chosen as $E=\{S_d/\sqrt d\ge a_k\}$, and a Chernoff bound gives
\[
 \mu_d(E)\le e^{-a_k^2/2}
 =\exp\{-2k-2k^{1/3}-\tfrac12k^{-1/3}\}
 \le\exp\{-\Lambda_{d,k}-2k^{1/3}\}.
\]
The $k^{1/3}$ correction comes from the cross term between the leading cutoff $2\sqrt k$ and its $k^{-1/6}$ displacement.
This proves sharpness at one fixed leakage level along sufficiently large dimensions.

\subsection{Limitations and open questions}\label{sec:consequences}
\noindent\textbf{Larger model dimensions.} At fixed confidence, determining the threshold to $O(k^{1/3})$ precision beyond $m\le\Mhard$ remains open.
Our bounds still identify the leading exponent when $k/d\to q\in(0,1/2)$ and $\log m=o(d)$ (see Appendix~\ref{app:modeldimscaling}).
When $\log m$ is comparable to $d$, how does the model dimension affect the leading exponent of the sampling threshold? Here, our lower-bound construction faces a limitation: concentrating an entire $m$-dimensional subspace on one set $E$ requires $\mu_d(E)\ge(1-\rho)m/\dim\cP_{d,k}$. Thus larger subspaces may require larger concentration sets, weakening the sampling lower bound obtained by this approach.

\medskip
\noindent\textbf{Sharpness and uniformity.} Our sharpness result concerns one fixed leakage level and sufficiently large ambient dimensions.
Optimality of the uncertainty remainder for every prescribed leakage, optimality of the $O(k^{1/3})$ remainder in the regression threshold, and uniform estimates as $k/d$ approaches $1/2$ from below remain open.

\medskip
\noindent\textbf{Geometry of individual subspaces.} Another question is which geometric properties of a given subspace $V$ determine its sample threshold. Such criteria should distinguish models requiring exponentially many samples in $k$ from spans of Walsh characters, which attain the parametric rate after $O(m(\log m+t))$ samples; see \Cref{app:cap}.

\section{From moment bounds to stable sampling}\label{sec:clipping}

\subsection{The Kirshner--Samorodnitsky moment bound}
The following sharp moment bound quantifies how the degree and ambient dimension limit the contribution of large polynomial values to the total energy.

First, for $p\ge2$ and $0<q<1/2$, let $z\in(0,1)$ be the unique solution of
\begin{equation*}\label{eq:KS-stationary-z}
 q=
 \frac{z\bigl((1+z)^{p-1}-(1-z)^{p-1}\bigr)}
 {(1+z)^p+(1-z)^p}
\end{equation*}
and define
\begin{equation}\label{eq:KS-rate}
 F(p,q)=\log\frac{(1+z)^p+(1-z)^p}{2}-pq\log z-\frac p2\HH(q).
\end{equation}
At the endpoints, set $F(p,0)=0$ and $F(p,1/2)=(p/2-1)\log2$, the respective limits of the formula.

The moment bound stated below is \cite[Corollary 1.4]{KS21}.
\begin{theorem}[Kirshner--Samorodnitsky~\cite{KS21}]\label[theorem]{thm:KS}
For $p\ge2$, $0\le k\le d/2$, and nonzero $f\in\cP_{d,k}$,
\begin{equation*}
 \frac{\E[|f|^p]}{(\E[f^2])^{p/2}}
 \le \exp\bigl(dF(p,k/d)\bigr),
\end{equation*}
with $F$ defined in \eqref{eq:KS-rate}.
\end{theorem}

\subsection{Cancellation at the second moment}
To make the complicated implicit rate \(F(p,q)\) easier to use, we write \(p=2+2s\) and study small \(s\ge0\), so that the moment bound controls deviations from the normalized second moment.
The next lemma identifies $\Psi(q)$ as the linear term near $p=2$, with no quadratic error. 
\begin{lemma}[Cancellation at the second moment term]\label[lemma]{lem:cubic}
For every $q_0<1/2$, there is $C_0<\infty$ such that, for
$0<q\le q_0$ and $0\le s\le1/2$,
\begin{equation}\label{eq:cubic-rate}
 F(2+2s,q)\le s\Psi(q)+C_0qs^3.
\end{equation}
Consequently, if $\E[h^2]=1$ and $\deg h\le k\le q_0d$, then
\begin{equation}\label{eq:cubic-moment}
 \log \E[|h|^{2+2s}]\le s\Lambda_{d,k}+C_0ks^3.
\end{equation}
\end{lemma}
The proof combines a hyperbolic parametrization with a carefully chosen variational comparison. A Taylor expansion then reveals the cancellation of the quadratic term and yields a uniform cubic remainder. See Appendix~\ref{app:moment-details} for the details.

\subsection{From moments to clipping}
The remaining step is a standard moment-to-tail argument, which exploits that a moment slightly above two bounds the energy removed by truncation. The resulting clipping upper bound matches that of the lower bound witness in \Cref{lem:scalar-witness}, including a remainder of the same order $k^{1/3}$ in the exponent, thereby establishing the following result.
 \begin{proposition}[One-dimensional clipping]\label[proposition]{prop:clipping}\label[proposition]{prop:Tupper}
 Fix $q_0\in(0,1/2)$ and $\eta\in(0,1)$. For integers $1\le k\le q_0d$,
 \[
  |\log T_{d,k}(\eta)-\Lambda_{d,k}|\le C_{\eta,q_0}k^{1/3}.
 \]
 \end{proposition}

\begin{proof}
\emph{Upper bound.}
\emph{Step 1: bounding the energy removed by clipping.}
Let $\E[h^2]=1$, $L_\eta=\log(1/(1-\eta))$, and $0<s\le1/2$. For $z\ge0$,
\[
 (z-B)_+
\le z\,\mathbf1_{\{z>B\}}
\le z\left(\frac zB\right)^s
= z^{1+s}B^{-s}.
\]
Using \Cref{lem:cubic},
\(
 \E[(h^2-B)_+]
 \le B^{-s}\E[|h|^{2+2s}]
 \le\exp\bigl(-s(\log B-\Lambda_{d,k})+C_0ks^3\bigr).
\)

\emph{Step 2: balancing the two error terms.}
To make the last exponent at most $-L_\eta$, it suffices to take $\log B-\Lambda_{d,k}=C_0ks^2+L_\eta/s$. Balancing these terms suggests $s$ of order $k^{-1/3}$. Choose
\begin{equation}\label{eq:schoice}
 s=\frac12\wedge
        \left(\frac{L_\eta}{2C_0k}\right)^{1/3},
 \qquad
 \log B=\Lambda_{d,k}+C_0ks^2+\frac{L_\eta}{s}.
\end{equation}
Then $\E[(h^2-B)_+]\le1-\eta$, so $\E[h^2\wedge B]\ge\eta$. If the second value in the minimum in \eqref{eq:schoice} is selected, the two error terms are each $O_{q_0,\eta}(k^{1/3})$. Otherwise, $k$ is bounded in terms of $q_0,\eta$, and the same conclusion holds after increasing the constant. Taking the supremum proves the claim.

\medskip
\noindent\emph{Lower bound.}
Fix the retained fraction $\eta$. For large $k$, apply \Cref{lem:scalar-witness} with $D=d$, $\ell=k$, and leakage $\rho=\eta/2$. Let $E=\{S_d\ge u\}$ be its small ball. Outside $E$, the polynomial has at most $\eta/2$ energy; inside $E$, clipping at $B$ leaves at most
$B$ at each input. Hence
\[
 \E[p^2\wedge B]
 \le\E[p^2\1_{E^c}]+B\mu_d(E)
 \le\eta/2+B\mu_d(E).
\]
Every $B<\eta/(2\mu_d(E))$ therefore retains less than $\eta$ energy.
By the definitions of $b_\eta$ and $T_{d,k}(\eta)$,
\[
 T_{d,k}(\eta)\ge b_\eta(p)
 \ge\frac{\eta}{2\mu_d(E)}
 \ge\frac\eta2\exp\{\Lambda_{d,k}-C_{\eta,q_0}k^{1/3}\}.
\]
For bounded $k$, the constant polynomial gives $T_{d,k}(\eta)\ge\eta$. Since $\Lambda_{d,k}\le2k$, increasing $C_{\eta,q_0}$ ensures $\log T_{d,k}(\eta)\ge \Lambda_{d,k}-C_{\eta,q_0}k^{1/3}$ in this case as well.
\end{proof}

\subsection{Uniform empirical covariance control from clipping}
We now convert the clipping estimate into a uniform lower bound on empirical energy. The argument combines three standard techniques: symmetrization, contraction, and Bousquet's inequality; see Appendix~\ref{app:tools}.

\begin{proposition}[Clipping gives stable sampling]\label[proposition]{prop:clip-cov}
Let $V$ have dimension $m$. Suppose
$\E[h^2\wedge B]\ge\alpha$ for every unit $h\in V$, where
$B\ge1$ and $0<c<\alpha<1$. Then
\[
 n\ge C_{\alpha,c}B(m+t)
 \quad\Longrightarrow\quad
 \Pp(\hSigma_V\succeq cI_m)\ge1-e^{-t}.
\]
\end{proposition}
\begin{proof}
Let $\psi_B(u)=u^2\wedge B$, and let $\mathcal S=\{h\in V:\E[h^2]=1\}$. This function is $2\sqrt B$-Lipschitz and vanishes at zero. Define
\[
 D_B=\sup_{h\in\mathcal S}\bigl(\E[\psi_B(h)]-\E_n[\psi_B(h)]\bigr),
 \qquad \overline D_B=\max\{D_B,0\}.
\]
Since $\overline D_B\le
\sup_{h\in\mathcal S}\left|\E[\psi_B(h)]-\E_n[\psi_B(h)]\right|$, Lemma~\ref{lem:symmetrization-contraction}(i) (symmetrization), followed by part (ii) (contraction) conditionally on the inputs with
$\varphi=\psi_B$ and $L=2\sqrt B$, yields
\[
 \E\overline D_B
 \le 8\sqrt B\,
 \E\sup_{h\in\mathcal S}
 \left|\frac1n\sum_i\varepsilon_i h(X_i)\right|
 =8\sqrt B\,
 \E\left\|\frac1n\sum_i\varepsilon_i\Phi_V(X_i)\right\|_2
 \le8\sqrt{\frac{Bm}{n}}.
\]
The last step uses Jensen's inequality and
\(
 \E\left\|\frac1n\sum_i\varepsilon_i\Phi_V(X_i)\right\|_2^2
 =\frac1n\E\|\Phi_V(X)\|_2^2=m/n.
\)
For a unit $h$, the centered function
\[
 g_h=\frac{\E[\psi_B(h)]-\psi_B(h)}B
\]
has absolute value at most $1$ and variance at most $1/B$ since
\(
 \E[\psi_B(h)^2]\le B\E[\psi_B(h)]\le B\E[h^2]=B.
\)

Next, we apply \Cref{thm:bousquet} (Bousquet's inequality) to these functions and the zero function. Its supremum is $Z=n\overline D_B/B$. After rescaling, with probability at least $1-e^{-t}$,
\begin{align*}
 \overline D_B
 &\le \E\overline D_B
 +\sqrt{\frac{2Bt}{n}\bigl(1+2\E\overline D_B\bigr)}
 +\frac{Bt}{3n}.
\end{align*}
When $n\ge CB(m+t)$ with $C$ large, the expectation is at most one so right side is bounded by
\[
 C'\left(\sqrt{\frac{B(m+t)}n}
                +\frac{B(m+t)}n\right).
\]
We can choose the sample-size constant so this is at most $\alpha-c$. On that event, for every $h\in\mathcal S$,
\[
 \E_n[h^2]
 \ge \E_n[\psi_B(h)]
 \ge \E[\psi_B(h)]-\overline D_B
 \ge\alpha-(\alpha-c)=c.
\]
For any nonzero $h\in V$, apply this bound to $h/\|h\|_2$ and multiply by $\|h\|_2^2$. The zero function satisfies the inequality trivially.

For $h=a^\top\Phi_V$, orthonormality gives
$\E[h^2]=\|a\|_2^2$, while
$\E_n[h^2]=a^\top\widehat\Sigma_Va$.
Thus the simultaneous bound $\E_n[h^2]\ge c\E[h^2]$
for all $h\in V$ is equivalent to
\[
 a^\top\widehat\Sigma_Va\ge c\|a\|_2^2
 \qquad(a\in\R^m),
\]
which is precisely $\widehat\Sigma_V\succeq cI_m$.
\end{proof}

\section{A polynomial concentrated on a Hamming ball}\label{sec:window}
In this section, we construct a spatially concentrated polynomial in the Krawtchouk basis, which is defined as 
\begin{equation}\label{eq:krawbasis}
 e_j(x)=\binom Dj^{-1/2}\sum_{|S|=j}\chi_S(x),
 \qquad 0\le j\le D.
\end{equation}

To guide the construction, we specify four requirements:
\begin{itemize}
\item \textbf{Radial symmetry:} $p$ depends only on $S_D$.
\item \textbf{Bounded degree:} $\deg p\le\ell$ for some appropriately chosen $\ell$.
\item \textbf{Unit norm:} $\E[p^2]=1$, so $p^2\,d\mu_D$ is a probability measure describing how the polynomial's energy is distributed. This lets us apply concentration bounds directly to its energy.
\item \textbf{Concentration at a high threshold:} for a prescribed leakage $\rho$,
\(
 \E[p^2\1_{\{S_D<u\}}]\le\rho,
\)
with $u$ large enough that $E_u$ has the required small probability.
\end{itemize}

\paragraph{Multiplication measures concentration.}
For a unit-norm polynomial $p$, we define
\[
 \eps^2=\|(S_D-\tau)p\|_2^2
 =\E[(S_D-\tau)^2p^2]
\]
to measure where $p$ concentrates its energy. If $\eps$ is small,
most energy lies where $S_D$ is close to $\tau$.
In particular, Markov's inequality gives
\[
 \E[p^2\1_{\{S_D<\tau-a\}}]\le\frac{\eps^2}{a^2}
 \qquad(a>0).
\]
Thus, the Hamming ball $\{S_D\ge\tau-a\}$ carries at least
$1-\eps^2/a^2$ of the energy.
We want, simultaneously, a large center $\tau$ and a small multiplication residual $\eps$. 
The following recurrence allows us to express this objective in terms of the coefficients $c_j$.

\begin{lemma}[Radial recurrence]\label[lemma]{lem:recurrence}
Define $a_j=\sqrt{(j+1)(D-j)}$, with $a_{-1}=a_D=0$, let $\{e_j\}_{j=0}^D$ denote the Krawtchouk basis as defined in \eqref{eq:krawbasis}, and let $e_{-1}=e_{D+1}=0$. Then $e_j$ is a polynomial of degree exactly $j$ in $S_D$, and
\begin{equation*}\label{eq:recurrence}
 S_De_j=a_je_{j+1}+a_{j-1}e_{j-1}.
\end{equation*}
\end{lemma}
\begin{proof}
Multiplying a Walsh character by $x_i$ adds $i$ to its index set if absent and removes it if present. Each resulting set of size $j+1$ occurs $j+1$ times, and each set of size $j-1$ occurs $D-j+1$ times. Normalizing gives the claimed recurrence. Starting with $e_0=1$, solving successively for $e_{j+1}$ proves the degree assertion.
\end{proof}

Consequently, multiplying $p=\sum_{j=0}^{\ell}c_je_j$ by $S_D$ transforms its coefficient vector according to the linear map
\[
 (J_Dc)_j=a_{j-1}c_{j-1}+a_jc_{j+1}.
\]
Here $c_j=0$ outside $0,\ldots,\ell$. 
Orthonormality therefore gives the identity
\[
 \epsilon^2
 =\sum_{j=0}^D\bigl((J_D-\tau I)c\bigr)_j^2
 =\|(J_D-\tau I)c\|_2^2
\]
which directly relates energy concentration to the coefficients in the Krawtchouk basis.

\paragraph{Choosing the coefficients $c_j$.}
We now seek a unit vector supported on degrees at most $\ell$ for which $\|(J_D-\tau I)c\|_2$ is small. In other words, we have to pick $c$ such that it is an \emph{approximate eigenvector} of $J_D$ with a large approximate eigenvalue $\tau$. Let $L$ denote the number of nonzero coefficients. Three constraints guide our construction:
\begin{itemize}
\item \textbf{Use consecutive degrees.} The mapping $c \mapsto J_Dc$ couples only adjacent indices. Consecutive coefficients allow the two neighboring contributions to reinforce each other while gaps would interrupt this reinforcement.
\item \textbf{Place non-zero coefficients only near $\ell$.} The weights $a_j$ monotonically increase until index $D/2$. Over a short interval near $\ell$, both neighboring weights are close to $a_*=\sqrt{\ell(D-\ell)}$. If $c_{j-1}+c_{j+1}\approx2c_j$, then $(J_Dc)_j\approx2a_*c_j$. This suggests taking $\tau=2a_*$. When $\ell\ll D$, this concentration location is approximately $2\sqrt{D\ell}$, well above the typical scale $\sqrt D$ of $S_D$ when $\ell$ is large. A cutoff close to $\tau=2\sqrt{\ell(D-\ell)}$ therefore defines a small-probability set on which we aim to concentrate the energy.
\item \textbf{Taper the coefficients to zero.}
A natural first attempt would be to assign equal coefficients $L^{-1/2}$ to $L$ consecutive degrees ending at $\ell$, and zero elsewhere. This gives $\sum_j c_j^2=1$, but the abrupt cutoff creates a large error since $c_{\ell+1}=0$,
\(
 ((J_D-\tau I)c)_{\ell+1}
 =a_\ell c_\ell=a_\ell/\sqrt L.
\)
Thus, this single coordinate already contributes $a_\ell/\sqrt L$ to the residual norm. To avoid this error, we let the coefficients rise gradually from zero and then fall gradually back to zero.
The relevant measure of smoothness is the second difference
\[
 \Delta^2c_j
 :=(c_{j+1}-c_j)-(c_j-c_{j-1})=c_{j-1}-2c_j+c_{j+1},
\]
which measures how much the slope of the coefficient sequence
changes between adjacent indices. Small second differences
give precisely the relation $c_{j-1}+c_{j+1}\approx2c_j$
needed above. We require this smoothness also at the ends
of the interval, where the sequence joins the zero coefficients.
\end{itemize}
In summary, we are looking to construct a \emph{smooth window}: a coefficient sequence supported on $L$ consecutive degrees and tapered to zero at both ends, with small second differences. The next section constructs such a sequence. Increasing $L$ permits smaller second differences, but includes matrix weights farther from $a_*$. Balancing these two errors leads to our choice $L=\lfloor\ell^{1/3}\rfloor$.

\subsection{Choosing the coefficients and bounding the error}
The next lemma makes this plan quantitative.

\begin{lemma}[A concentrating polynomial]\label[lemma]{lem:window}
Fix $q_0<1/2$. For all sufficiently large $\ell\le q_0D$, there is a unit norm radial polynomial $p$ of exact degree $\ell$ such that
\begin{equation*}
 \|(S_D-2\sqrt{\ell(D-\ell)})p\|_2
 \le C_{q_0}\sqrt D\,\ell^{-1/6}.
\end{equation*}
\end{lemma}
\begin{proof}
We use the $L=\lfloor\ell^{1/3}\rfloor$ degrees $\ell-L+1,\ldots,\ell$. To make the coefficients small near the two ends of this interval, define
\[
 b(x)=\begin{cases}x^2(1-x)^2,&0\le x\le1,\\0,&\text{otherwise},\end{cases}
 \qquad
 c_{\ell-L+r}=Z_L^{-1/2}b\!\left(\frac r{L+1}\right)
 \quad(1\le r\le L),
\]
where $Z_L=\sum_{r=1}^L b(r/(L+1))^2$ normalizes $\sum_jc_j^2$ to one. All other coefficients are zero, including at integer indices outside $0,\ldots,D$. For large $\ell$, $2\le L\le\ell/4$.

We first bound the second differences of this sequence. The function $b$ is bounded above and below by positive constants on $[1/4,3/4]$. A fixed positive fraction of the $L$ sampling points lie in that interval, so $cL\le Z_L\le CL$ for positive constants $c,C$. In particular, the normalization contributes a factor of order $L^{-1/2}$.

Both $b$ and its first derivative vanish at $0$ and $1$. Extending $b$ by zero therefore leaves its derivative continuous, and we have $|b'(u)-b'(v)|\le C|u-v|$ for all $u,v$ since inside $[0,1]$ its second derivative is bounded, and outside it the derivative is zero. For the sampling step $\delta_L=1/(L+1)$, it follows that
\[
 |b(x+\delta_L)-2b(x)+b(x-\delta_L)|
 =\left|\int_0^{\delta_L}
       \bigl(b'(x+t)-b'(x-\delta_L+t)\bigr)\,dt\right|
 \le C\delta_L^2.
\]
Consequently, for $\Delta^2c_j=c_{j+1}-2c_j+c_{j-1}$, every second difference has magnitude at most $CZ_L^{-1/2}\delta_L^2\le CL^{-5/2}$. Only the $L+2$ indices $\ell-L,\ldots,\ell+1$ can contribute, since at all other indices the three coefficients are zero. Hence
\[
 \|\Delta^2c\|_{\ell_2(\mathbb Z)}
 =\left(\sum_{j\in\mathbb Z}|\Delta^2c_j|^2\right)^{1/2}
 \le\sqrt{L+2}\,CL^{-5/2}\le CL^{-2}.
\]
Now, set $p=\sum_jc_je_j$, so $\|p\|_2=1$ and $c_\ell>0$. We separate the residual into two contributions: the second differences of $c$ and the variation of the weights $a_j$. Since $p=\sum_j c_j e_j$, the coefficient multiplying $e_j$ in $(S_D-2a_*)p$ is
\[
\begin{aligned}
 ((J_D-2a_*I)c)_j
 &=a_jc_{j+1}+a_{j-1}c_{j-1}-2a_*c_j\\
 &=a_*(c_{j-1}-2c_j+c_{j+1})
   +(a_j-a_*)c_{j+1}
   +(a_{j-1}-a_*)c_{j-1}\\
 &=a_*\Delta^2c_j
   +(a_j-a_*)c_{j+1}
   +(a_{j-1}-a_*)c_{j-1},
\end{aligned}
\]
where $a_*=\sqrt{\ell(D-\ell)}$. The first term is controlled by the smoothness estimate above. For the remaining terms, we only need to bound those where $c_{j+1}$ or $c_{j-1}$ is nonzero.  Each nonzero coefficient \(c_r\) contributes through the weights \(a_{r-1}\) and \(a_r\). Since \(|r-\ell|\le L+2\) on the support of \(c\), both weight indices lie within \(L+3\) of \(\ell\). 
Using $a_i^2=(i+1)(D-i)$ and $a_*^2=\ell(D-\ell)$,
\[
\begin{aligned}
 a_i^2-a_*^2
 &=(i-\ell)(D-i-\ell)+(D-i),\\
 |a_i^2-a_*^2|
 \le |i-\ell|&\,|D-i-\ell|+|D-i|
 \le D(|i-\ell|+1),
\end{aligned}
\]
where the last inequality uses $0\le i,\ell\le D$.

Moreover, $\ell\le q_0D$ implies $a_*\ge\sqrt{1-q_0}\sqrt{D\ell}$. Hence, whenever \(c_{i+1}\ne0\) or \(c_i\ne0\),
\[
 |a_i-a_*|
 =\frac{|a_i^2-a_*^2|}{a_i+a_*}
 \le
 \frac{D(|i-\ell|+1)}
      {\sqrt{1-q_0}\sqrt{D\ell}}
\le C_{q_0}\sqrt{\frac D\ell}\,(L+4).
\]

Thus, using $L\asymp\ell^{1/3}$,
\[
 \|(J_D-2a_*I)c\|_2
 \le C_{q_0}\left(
 \frac{\sqrt{D\ell}}{L^2}
 +\sqrt{\frac D\ell}(L+4)\right)
 \le C_{q_0}\sqrt D\,\ell^{-1/6}.
\]
Finally, orthogonality of the $e_j$ proves the statement of the lemma.
\end{proof}

\subsection{From an approximate eigenvector to small-set concentration}
\Cref{lem:window} places the energy of $p$ \emph{near} $\tau=2\sqrt{\ell(D-\ell)}$, but does not determine how much lies \emph{above} $\tau$. The ball $\{S_D\ge\tau\}$ could therefore miss substantial energy. Instead, taking the ball $\{S_D\ge\tau-a\}$ with appropriately chosen $a>0$, we can upper bound the leakage by a constant $\rho$.
However, we need to be careful as to not increase the probability mass of the ball too much when lowering the cutoff. The next lemma shows that for $a=M_{\rho,q_0}\sqrt D\,\ell^{-1/6}$, this increase costs at most $C_{\rho,q_0}\ell^{1/3}$ in the exponent.

\begin{lemma}[One-dimensional concentration witness]\label[lemma]{lem:scalar-witness}
Fix $\rho\in(0,1)$ and $q_0\in(0,1/2)$. There are constants $M_{\rho,q_0},C_{\rho,q_0}>0$ such that, for all sufficiently large $\ell\le q_0D$, there is a unit radial polynomial $p$ of exact degree $\ell$ for which the cutoff $\tau=2\sqrt{\ell(D-\ell)}~,~u=\tau-M_{\rho,q_0}\sqrt D\,\ell^{-1/6}>0$
satisfies
\[
 \E[p^2\1_{\{S_D<u\}}]\le\rho,\qquad
 \mu_D(S_D\ge u)
 \le\exp\{-D\Psi(\ell/D)+C_{\rho,q_0}\ell^{1/3}\}.
\]
\end{lemma}
\begin{proof}
\emph{Step 1: retaining the energy.}
Take $p$ from \Cref{lem:window} and set $a=M_{\rho,q_0}\sqrt D\,\ell^{-1/6}$, so $u=\tau-a$. On $\{S_D<u\}$, the distance from $S_D$ to $\tau$ exceeds $a$.
Consequently,
\[
 \E[p^2\1_{\{S_D<u\}}]
 \le\frac{\E[(S_D-\tau)^2p^2]}{a^2}
 \le\frac{C_{q_0}^2D\ell^{-1/3}}
          {M_{\rho,q_0}^2D\ell^{-1/3}}
 \le\rho
\]
if $M_{\rho,q_0}\ge C_{q_0}/\sqrt\rho$. Also $\tau\ge2\sqrt{(1-q_0)D\ell}$, so $a/\tau=O_{\rho,q_0}(\ell^{-2/3})$. Thus $0<u<\tau<D$ for large $\ell$, so both cutoffs lie in the upper-tail range where the subsequent binomial probability bound applies.

\emph{Step 2: bounding the uniform probability.}
For $0\le z<1$, write
$
 I(z)=\log2-\HH((1-z)/2)~, ~I'(z)=\artanh z.
$
The binomial Chernoff bound (e.g., \cite[Section~2.2]{BLM13}) states that $\mu_D(S_D\ge u)\le\exp\{-DI(u/D)\}$. At the cutoff $\tau$, the rate is exactly
\(
 DI(\tau/D)=D\Psi(\ell/D).
\)
We therefore need to bound the loss in this rate when the cutoff moves from $\tau$ down to $u$. Since $I'$ is increasing, the gap in the exponent is bounded as
\begin{equation}\label{eq:scalar-rate-loss}
 D\Psi(\ell/D)-DI(u/D)
 =D\int_{u/D}^{\tau/D}I'(z)\,dz
 \le(\tau-u)I'(\tau/D).
\end{equation}
The assumption $\ell/D\le q_0<1/2$ gives
$\tau/D\le2\sqrt{q_0(1-q_0)}<1$. On this interval,
$\artanh z\le C_{q_0}z$, and therefore
\(
 I'(\tau/D)\le C_{q_0}\sqrt{\ell/D}.
\)
Combining the last two bounds yields
$D\Psi(\ell/D)-DI(u/D)
 \le C_{\rho,q_0}
 \sqrt D\,\ell^{-1/6}
 \sqrt{\ell/D}
 =C_{\rho,q_0}\ell^{1/3}.$
Substituting into the Chernoff bound proves the second assertion.
\end{proof}

\section{From one polynomial to a concentrated subspace}\label{sec:common}
In this section we extend the spatially concentrated polynomial of the preceding section to many directions via multiplication by elements in $\cH_{d,r}$, the space of homogeneous multilinear degree-$r$ harmonic polynomials. These are characterized by the identity $\sum_{i=1}^d\partial_i f=0$.
For example, $\cH_{d,0}$ consists of constants, while
$
\cH_{d,1}=\left\{\sum_i v_i x_i:\sum_i v_i=0\right\}.
$

\subsection{Uniform energy control across harmonic directions}
All unit polynomials in $\cH_{d,r}$ have the same distribution of energy across slices. A shared radial factor therefore changes this distribution identically in every direction. Moreover, different harmonic degrees remain orthogonal under radial weighting. These facts follow from the identity below, a direct specialization of \cite[Theorem~3.10]{FM19}. See Appendix~\ref{app:harmonic-details} for the proof.

\begin{lemma}[Radial-weight identity]\label[lemma]{lem:harmonic-radial}
For integers $0\le r,s\le\lfloor d/2\rfloor$, $f\in\cH_{d,r}$, $g\in\cH_{d,s}$, and any real function $H$ of the
coordinate sum,
\begin{equation}\label{eq:harmonic-radial}
 \E_{\mu_d}[fgH(S_d)]
 =\begin{cases}
 0,&r\ne s,\\
 \langle f,g\rangle_d\,\E_{\mu_{d-2r}}[H(S_{d-2r})],&r=s.
 \end{cases}
\end{equation}
Here $\langle f,g\rangle_d=\E_{\mu_d}[fg]$ and $S_0=0$.
\end{lemma}

To get an intuition as to why $d-2r$ appears, consider the unit harmonic polynomial
\[
 Q_r(x)=\prod_{j=1}^r\frac{x_{2j-1}-x_{2j}}{\sqrt2}.
\]
Taking $f=g=Q_r$ introduces the weight $Q_r^2$ in the uniform expectation. This weight vanishes unless $x_{2j-1}=-x_{2j}$ for every $j\le r$. Each pair contributes zero to $S_d$, leaving the sum of $d-2r$ independent signs. The lemma says that the same law of the coordinate sum results from weighting by $f^2/\|f\|_2^2$ for \emph{every} nonzero $f\in\cH_{d,r}$.

In particular, choose $p_r$ with unit norm under $S_{d-2r}$ and energy at most $\rho$ below a cutoff $u$. Apply the identity with $H(s)=p_r(s)^2$ and with $H(s)=p_r(s)^2\1_{\{s<u\}}$, respectively. For every nonzero $h_f=f p_r(S_d)$, this yields
\[
 \|h_f\|_2^2=\|f\|_2^2,
 \qquad
 \frac{\E[h_f^2\1_{\{S_d<u\}}]}{\|h_f\|_2^2}
 =\E[p_r(S_{d-2r})^2\1_{\{S_{d-2r}<u\}}]\le\rho.
\]
Thus multiplication by the radial factor preserves dimension and transfers the one-dimensional concentration bound to the entire space $W_r=\{f p_r(S_d):f\in\cH_{d,r}\}$, with no dimension factor in the leakage. Hence we call this construction a \emph{harmonic lift}. 
Different harmonic degrees remain orthogonal even after radial weighting. Consequently, if $h_r\in W_r$ and $h_s\in W_s$ with $r\ne s$, then
\[
 \E[h_rh_s]=0,
 \qquad
 \E[h_rh_s\1_{E^c}]=0
\]
for any Hamming ball $E$ centered at the all-plus vertex. Thus, for $h=\sum_{r=0}^R h_r$, the cross terms vanish both in its total energy and in its energy outside $E$:
\[
 \E[h^2]=\sum_{r=0}^R\E[h_r^2],
 \qquad
 \E[h^2\1_{E^c}]=\sum_{r=0}^R\E[h_r^2\1_{E^c}].
\]
If we choose one cutoff so that $\E[h_r^2\1_{E^c}]\le\rho\E[h_r^2]$ for every $r$ and every $h_r\in W_r$, summing these inequalities gives the same concentration bound for every function in $W_0\oplus\cdots\oplus W_R$.

Combining harmonic degrees $0\le r\le R$ supplies $\binom dR$ concentrated directions, but leaves only $k-r$ degrees for the radial factor in sector $r$. Its effective coordinate sum has $d-2r$ terms, so the one-dimensional construction has center $2\sqrt{(k-r)(d-k-r)}$. This center decreases with $r$. The common cutoff must therefore lie below the smallest center, attained at $R$. We define that center and its binomial tail exponent as
\begin{equation*}
 \tau_R=2\sqrt{(k-R)(d-k-R)},\qquad
 \mathcal E_{d,k}(R)=dI(\tau_R/d),\qquad 0\le R\le k.
\end{equation*}
Thus $\mathcal E_{d,k}(R)$ determines the leading exponential decay of the ball’s probability after reserving degree $R$ for the harmonic factors, before adjusting the cutoff to ensure energy concentration. At $R=0$ it equals $\Lambda_{d,k}$, and increasing $R$ trades a larger subspace for a smaller concentration exponent. Explicitly parametrizing this tradeoff will allow for the subsequent optimized choice $R=\lfloor k^{1/3}\rfloor$.

We choose $\deg p_r\le k-r$, so each product has Boolean degree at most $k$. The available number of harmonic directions then is \cite[Corollary~3.9(b)]{FM19}:
\begin{equation}\label{eq:harmonic-dims}
 \dim\cH_{d,r}=\binom dr-\binom d{r-1},\qquad \binom d{-1}=0.
\end{equation}

\begin{proposition}[Harmonic lift]\label[proposition]{prop:harmonic-lift}
Fix $\rho\in(0,1)$ and $q_0\in(0,1/2)$. For integers $d,k,R$ with $1\le k\le q_0d$ and $0\le R\le k$, there are $W_{\le R}\subseteq\cP_{d,k}$ of dimension $\binom dR$ and a permutation-invariant $E$ such that, with $K=k-R$,
\begin{equation*}
 \mu_d(E)\le e^{-\mathcal E_{d,k}(R)+C_{\rho,q_0}K^{1/3}},
 \qquad \E[h^2\1_{E^c}]\le\rho \E[h^2]\quad(h\in W_{\le R}).
\end{equation*}
Consequently,
\begin{equation}\label{eq:harmonic-cost}
 \mu_d(E)\le e^{-\Lambda_{d,k}+C_{\rho,q_0}(R+k^{1/3})}.
\end{equation}
\end{proposition}
\begin{proof}
Write $M=M_{\rho,q_0}$ from \Cref{lem:scalar-witness}. Fix an integer $K_0=K_0(\rho,q_0)\ge1$ above that lemma's degree threshold and large enough that $M K_0^{-2/3}<2\sqrt{1-2q_0}$.

\emph{Case (i): $K\ge K_0$.}

\emph{Step 1: choosing a common ball.}
For $0\le r\le R$, put $D_r=d-2r$ and $K_r=k-r$. Since $K_r\ge K_0$ and $K_r/D_r\le k/d\le q_0$, \Cref{lem:scalar-witness} provides a polynomial $p_r$ of degree $K_r$, with unit norm under $S_{D_r}$ and leakage at most $\rho$ below
\[
 u_r=2\sqrt{(k-r)(d-k-r)}-M\sqrt{D_r}\,K_r^{-1/6}.
\]
Set $u=\tau_R-M\sqrt d\,K^{-1/6}$ and $E=\{S_d\ge u\}$. The centers decrease with $r$, while $\sqrt{D_r}K_r^{-1/6}\le\sqrt dK^{-1/6}$, so $u\le u_r$ for all $r$. Moreover, $\tau_R\ge2\sqrt{(1-2q_0)dK}$; our choice of $K_0$ therefore ensures $0<u<\tau_R<d$.

\emph{Step 2: combining directions without increasing leakage.}
Define $W_r=\{f p_r(S_d):f\in\cH_{d,r}\}$, whose elements have degree at most $r+K_r=k$. By \Cref{lem:harmonic-radial}, multiplication by $p_r(S_d)$ preserves norms, and $\E[h_r^2\1_{E^c}]\le\rho\|h_r\|_2^2$ for all $h_r\in W_r$.
 
To extend this bound to $h=\sum_r h_r$, we must also control the cross terms in both $\|h\|_2^2$ and $\E[h^2\1_{E^c}]$. For $r\ne s$, apply \Cref{lem:harmonic-radial} with the weights $H(t)=p_r(t)p_s(t)$ and $H(t)=p_r(t)p_s(t)\1_{\{t<u\}}$, respectively.
This yields
\(
 \E[h_rh_s]=\E[h_rh_s\1_{E^c}]=0.
\)
Hence, by \eqref{eq:harmonic-dims}, $W_{\le R}=\bigoplus_{r=0}^R W_r$ has dimension
$\sum_{r=0}^R[\binom dr-\binom d{r-1}]=\binom dR$, and
\[
 \E[h^2\1_{E^c}]
 =\sum_r\E[h_r^2\1_{E^c}]
 \le\rho\sum_r\|h_r\|_2^2=\rho\|h\|_2^2.
\]

\emph{Step 3: bounding the probability of the ball.}
Use the integral estimate in \eqref{eq:scalar-rate-loss}, now with $D=d$ and center $\tau_R$. Since $\tau_R/d\le2\sqrt{q_0(1-q_0)}<1$ and $I'(\tau_R/d)\le C_{q_0}\sqrt{K/d}$, it gives
\[
 \mathcal E_{d,k}(R)-dI(u/d)
 \le(\tau_R-u)I'(\tau_R/d)
 \le C_{\rho,q_0}\sqrt d\,K^{-1/6}\sqrt{K/d}
 =C_{\rho,q_0}K^{1/3}.
\]
The binomial Chernoff bound (see, e.g., \cite[Section~2.2]{BLM13}) $\mu_d(E)\le e^{-dI(u/d)}$ now proves the asserted probability bound.

\emph{Case (ii): $0\le K<K_0$.}
Take $E$ to be the whole cube and $W_{\le R}$ to be the span of all degree-$R$ Walsh characters. This space has dimension $\binom dR$, lies in $\cP_{d,k}$, and has zero leakage. To check the probability bound, note that $\tau_R\le2\sqrt{K(d-K)}$ and $\Psi(q)\le2q$ give $0\le\mathcal E_{d,k}(R)\le d\Psi(K/d)\le2K$. For $1\le K<K_0$, choosing $C_{\rho,q_0}\ge2K_0^{2/3}$ makes $-\mathcal E_{d,k}(R)+C_{\rho,q_0}K^{1/3}\ge0$, so the claimed upper bound is at least $\mu_d(E)=1$. For $K=0$, both exponents are zero.

Finally, for $z_R=\tau_R/d$ and by continuity at $R=k$,
\[
 \left|\frac{d}{dR}\mathcal E_{d,k}(R)\right|
 =2(1-2R/d)\frac{\artanh z_R}{z_R}\le C_{q_0}.
\]
Integrating gives $\mathcal E_{d,k}(R)\ge \Lambda_{d,k}-C_{q_0}R$, which proves \eqref{eq:harmonic-cost}.
\end{proof}

\subsection{Proof of \Cref{thm:common}}
\begin{proof}
\emph{Step 1: the universal lower bound on $\mu_d(E)$.}
Let $h\ne0$ satisfy $\E[h^2\1_{E^c}]\le\rho\|h\|_2^2$ and rescale it to $\|h\|_2=1$. Set $\eta=(1+\rho)/2>\rho$. By \Cref{prop:Tupper}, clipping at $B=\exp\{\Lambda_{d,k}+C_{\eta,q_0}k^{1/3}\}$ retains at least $\eta$ energy. Clipping leaves at most $\rho$ energy outside $E$, while its contribution on $E$ is at most $B\mu_d(E)$. Hence
\[
 \eta\le\E[h^2\wedge B]
 \le\E[h^2\1_{E^c}]+B\mu_d(E)
 \le\rho+B\mu_d(E).
\]
It follows that $\mu_d(E)\ge(1-\rho)/(2B)$, giving the asserted bound after absorbing the fixed prefactor into $C_{\rho,q_0}$.

\emph{Step 2: a subspace attaining the bound up to the remainder.}
Apply \Cref{prop:harmonic-lift} with $R=\lfloor k^{1/3}\rfloor$. It supplies a space of dimension $\binom dR=\Mhard$ and one event $E$ with leakage at most $\rho$ for every function in that space. Finally,
\eqref{eq:harmonic-cost} gives
\[
 \mu_d(E)\le\exp\{-\Lambda_{d,k}+C_{\rho,q_0}(R+k^{1/3})\}
 \le\exp\{-\Lambda_{d,k}+C'_{\rho,q_0}k^{1/3}\}.
\]
\end{proof}

\begin{remark}[Shared small set is a Hamming ball]
For sufficiently large $k$, the choice $R=\lfloor k^{1/3}\rfloor$ gives $K=k-R\ge K_0(\rho,q_0)$. We therefore land in case (i) of the proof of \Cref{prop:harmonic-lift}, which constructs $E=\{S_d\ge u\}$ with $u>0$, rather than the whole-cube fallback of case (ii). This event is a Hamming ball, establishing the final assertion of \Cref{thm:common}.
\end{remark}

\section{Empirical eigenvalues and the sampling theorem}\label{sec:statistics-main}\label{sec:assembly}

The following lemma shows how joint energy concentration on a common small set forces many empirical eigenvalues to be small when too few observations fall in that set.

\begin{lemma}[Many small empirical eigenvalues]\label[lemma]{lem:rare-cov}
Let $p=\mu_d(E)\in(0,1/2]$. 

\smallskip
\noindent\textup{(i)} If $\E[\Phi(X)\Phi(X)^\top\1_{\{X\notin E\}}]\preceq\rho I_m$, $m\ge2$, and $n\le m/(128p)$, then with probability at least $15/16$, at least $\lfloor m/2\rfloor$ eigenvalues of $\hSigma$ are at most $128\rho$. 

\smallskip
\noindent\textup{(ii)} If $t\ge\log4$, $n\le t/(512p)$, and a fixed unit vector $v$ satisfies $\E[(v^\top\Phi)^2\1_{E^c}]\le\rho$, then with probability at least $(15/16)e^{-t/256}$ the smallest eigenvalue of $\hSigma$ is at most $32\rho$.
\end{lemma}
\begin{proof}
First, we split the empirical covariance according to where each observation falls:
\[
 \hSigma=\hSigma_E+\hSigma_{E^c},\qquad
 \hSigma_A=\frac1n\sum_{i=1}^n
 \Phi(X_i)\Phi(X_i)^\top\1_{\{X_i\in A\}},\quad A\in\{E,E^c\}.
\]
Both matrices are positive semidefinite.

\emph{Part (i): many small eigenvalues.}
Let $N_E=\sum_i\1_{\{X_i\in E\}}$. Then $\operatorname{rank}(\hSigma_E)\le N_E$, $\E[N_E]=np\le m/128$, and $\E[\Tr\widehat\Sigma_{E^c}]=\Tr\!\left(\E[\Phi(X)\Phi(X)^\top\1_{\{X\notin E\}}]\right)\le\Tr(\rho I_m)=m\rho$ by the matrix inequality assumption.

Since \(\E[N_E]=np\), Markov’s inequality gives
\[
\Pp\{N_E>m/4\}\le\frac{4np}{m},
\qquad
\Pp\{\Tr\widehat\Sigma_{E^c}>32m\rho\}
\le\frac1{32}.
\]
Thus a union bound yields
\[\Pp\{N_E\le m/4 ~\wedge~ \Tr\widehat\Sigma_{E^c}\le32m\rho\}\ge1-\Pp\{N_E>m/4\}-\Pp\{\Tr\widehat\Sigma_{E^c}>32m\rho\}\ge1-\frac1{32}-\frac1{32}=\frac{15}{16}.\]

Condition on the event on the left-hand side and set $L=\ker\hSigma_E$, so that $\dim L\ge m-\lfloor m/4\rfloor$. Let $P_L$ be orthogonal projection onto $L$. Consider the operator
$B_L=P_L\hSigma_{E^c}P_L$. It has the same quadratic form as $\hSigma_{E^c}$ on vectors in $L$, and $\Tr B_L\le\Tr\hSigma_{E^c}\le32m\rho$. Since its eigenvalues are nonnegative, at most $\lfloor m/4\rfloor$ can exceed $128\rho$. Since $B_L$ is symmetric and maps $L$ into itself, $L$ has an orthonormal basis of eigenvectors of $B_L$. Let $U$ be the span of those basis vectors whose eigenvalues are at most $128\rho$. Then
$
 \dim U\ge\dim L-\lfloor m/4\rfloor
 \ge\lfloor m/2\rfloor.
$

For $0\ne w\in U$,
\[
 \mathcal R_{\hSigma}(w)
 =\frac{w^\top\hSigma_{E^c}w}{\|w\|_2^2}
 =\frac{w^\top B_Lw}{\|w\|_2^2}\le128\rho,
\]
because $\hSigma_Ew=0$. The min--max principle now gives the claimed number of small eigenvalues of $\hSigma$.

\emph{Part (ii): one small eigenvalue.}
Fix a unit vector $v$ satisfying the assumption in (ii) and condition on $N_E=0$. Independence gives
\[
 \Pr(N_E=0)=(1-p)^n\ge e^{-2np}\ge e^{-t/256},
 \qquad
 \E[v^\top\hSigma v\mid N_E=0]
 =\frac{\E[(v^\top\Phi)^2\1_{E^c}]}{1-p}\le2\rho.
\]
Markov's inequality gives $v^\top\hSigma v\le32\rho$ with conditional probability at least $15/16$. Since $\lambda_{\min}(\hSigma)\le\mathcal R_{\hSigma}(v)=v^\top\hSigma v$, the unconditional probability is at least $(15/16)e^{-t/256}$.
\end{proof}

\subsection{Proof of \Cref{thm:main}}
\begin{proof}
Write $N=\Npar(d,k,m,A,\delta)$.

\emph{Step 1: upper bound.}
Apply \Cref{prop:Tupper} with a fixed retained fraction greater than $1/2$, followed by \Cref{prop:clip-cov} with $c=1/2$ and failure probability $\delta/2$. On the event $\{\widehat\Sigma_V\succeq I_m/2\}$, conditionally on $X_1,\ldots,X_n$, the ordinary least-squares estimator satisfies
\[
 \widehat\theta_{\mathrm{OLS}}-\theta
 \sim N\!\left(0,\frac{\sigma^2}{n}\widehat\Sigma_V^{-1}\right),
 \qquad
 \|\widehat\theta_{\mathrm{OLS}}-\theta\|_2^2
 \preceq_{\rm st}\frac{2\sigma^2}{n}\chi_m^2.
\]
\Cref{cor:ols} therefore gives the required error bound with constant $12\le A$. Thus
\[
 N\le C(m+t)e^{\Lambda_{d,k}+Ck^{1/3}}
\]
for every feasible $m$.

\emph{Step 2: lower bound.}
Choose the leakage required in \Cref{prop:rare-par}. For sufficiently large $k$, \Cref{thm:common} supplies a common concentration set of probability
\[
 0<p\le e^{-\Lambda_{d,k}+Ck^{1/3}}\le1/2.
\]
If $m\le\Mhard$, take an $m$-dimensional subspace of the concentrating space. If $t\ge m$, extend one concentrated unit vector to a feasible $m$-dimensional space. \Cref{prop:rare-par} rules out the required regression guarantee at every positive integer $n\le(m+t)/(1024p)$.

Consequently,
\[
 N\ge\frac{m+t}{1024p}
 \ge\frac{m+t}{1024}e^{\Lambda_{d,k}-Ck^{1/3}}.
\]
Here the first inequality also holds when its right-hand side is below one, since $N\ge1$. For bounded $k$, \Cref{lem:baseline} gives $N\ge c_0(m+t)$, which suffices because $\Lambda_{d,k}\le2k$. Taking logarithms and absorbing fixed constants into $Ck^{1/3}$ completes the proof.
\end{proof}

\section{Sharpness of the uncertainty remainder}\label{sec:airy}
We prove \Cref{prop:airy-sharp} by constructing a sequence of dimensions $d_k\ge k^3$ and subspaces of dimension $\binom{d_k}{\lfloor k^{1/3}\rfloor}$ such that every nonzero function in each subspace has at least a fixed fraction $\beta>0$ of its energy on a common set of probability at most $e^{-\Lambda_{d_k,k}-2k^{1/3}}$.
Thus the leakage is the fixed value $\rho_*=1-\beta$. Unlike the construction of Section~\ref{sec:window}, this argument does not allow an arbitrarily prescribed leakage.

The proof has three stages. We first construct a polynomial of degree at most \(N-1\) whose energy under the standard Gaussian measure has a fixed positive fraction concentrated on \(\{x>2\sqrt N\}\). For each sufficiently large degree $k$, we transfer that polynomial to a sufficiently large Boolean cube and apply the harmonic lift. Finally, a tail bound shows that the common concentrating set has the required probability.

\subsection{Convergence and limits}
Two different limits underlie the construction. We describe them below.

\medskip
\noindent\textbf{From the cube to Gaussian polynomials: fix the degree, let the ambient dimension grow.} Write the normalized radial Krawtchouk polynomial as
\[
 e_{D,j}(x)=\binom Dj^{-1/2}\sum_{|A|=j}\prod_{i\in A}x_i
           =P_{D,j}(S_D(x)/\sqrt D).
\]
For every fixed $j$, the coefficients of the univariate polynomial $P_{D,j}$ converge as $D\to\infty$ to those of $H_j(t)=\operatorname{He}_j(t)/\sqrt{j!}$, the degree-$j$ orthonormal Hermite polynomial for the standard Gaussian distribution. In particular, $P_{D,j}(t)\to H_j(t)$ uniformly on bounded intervals of $t$. This explains why Gaussian polynomials are useful models for radial polynomials on large cubes. Our passage from the Gaussian construction to the cube only needs the central limit theorem and moment convergence for a fixed polynomial.

\medskip
\noindent\textbf{From Hermite polynomials to the Airy kernel as the degree cutoff increases.} We combine Hermite degrees $0,\ldots,N-1$ into the kernel
$K_N$ defined below. The external asymptotic theorem we use evaluates this kernel at points
\[
 x_N(s)=2\sqrt N+sN^{-1/6},
\]
where $s$ remains in a fixed bounded interval as $N\to\infty$.
After multiplying the kernel by $N^{-1/6}$, its limit is the \emph{Airy kernel} in \eqref{eq:airy-limit}. Here it describes the limiting behavior of the Hermite kernel near the edge of its concentration region.

\subsection{A Gaussian polynomial with a fixed amount of tail energy}
\paragraph{Constructing a concentrated polynomial with the Hermite kernel.}
Let $\gamma=N(0,1)$ have density
$\varphi(x)=(2\pi)^{-1/2}e^{-x^2/2}$, and define
\[
 H_j(x)=\frac{\operatorname{He}_j(x)}{\sqrt{j!}},
 \qquad
 \operatorname{He}_j(x)=(-1)^j e^{x^2/2}
       \frac{d^j}{dx^j}e^{-x^2/2}.
\]
The $H_j$ are orthonormal in $L_2(\gamma)$. Set $\psi_j=H_j\sqrt{\varphi}$ and define the Hermite kernel
\[
 K_N(x,y)=\sum_{j=0}^{N-1}\psi_j(x)\psi_j(y).
\]
The Airy function and its associated kernel are defined, for real $s,t$, by
\[
 \operatorname{Ai}(s)
 =\frac1\pi\lim_{R\to\infty}\int_0^R\cos\!\left(\frac{u^3}{3}+su\right)\,du,
 \qquad
 K_{\mathrm{Ai}}(s,t)
 =\int_0^\infty\operatorname{Ai}(s+u)\operatorname{Ai}(t+u)\,du.
\]
To describe the Hermite kernel near $2\sqrt N$, put
\(
 x_N(s)=2\sqrt N+sN^{-1/6}.
\)
Then
\begin{equation}\label{eq:airy-limit}
 N^{-1/6}K_N(x_N(s),x_N(t))
 \longrightarrow K_{\mathrm{Ai}}(s,t)
\end{equation}
as $N\to\infty$, uniformly for $(s,t)$ in compact subsets of $\R^2$. This convergence is the Gaussian specialization of \cite[Theorem~1.1, $\beta=2$]{DG07}.

We center a kernel vector at $y_N=x_N(3)$, slightly beyond $2\sqrt N$. This is the standard reproducing-kernel construction: the coefficients are the basis values at the point where we want concentration. For $K=N-1$, define
\begin{equation*}
 q_K(x)=
 \frac{\sum_{j=0}^{K}H_j(x)H_j(y_N)}
      {(\sum_{j=0}^{K}H_j(y_N)^2)^{1/2}}.
\end{equation*}
Orthonormality gives $\deg q_K\le K$ and $\int q_K^2\,d\gamma=1$.
Moreover,
\[
 q_K(x)^2\varphi(x)
 =\frac{K_N(x,y_N)^2}{K_N(y_N,y_N)}.
\]
\paragraph{Concentrating a fixed fraction of energy on a scaled interval.}

We estimate the energy on the interval $[2\sqrt N+3N^{-1/6},\,2\sqrt N+4N^{-1/6}]$. Under the change of variables $x=2\sqrt N+sN^{-1/6}$, this becomes $s\in[3,4]$, where the uniform convergence in \eqref{eq:airy-limit} lets us pass to the limit inside the integral. The left endpoint is chosen far enough above $2\sqrt N$ to accommodate the degree reserved for the harmonic lift below. Any fixed right endpoint larger than $3$ would work. A positive lower bound on the energy in this interval suffices, since the interval lies within the tail event we will use.

To track the scaling, write $\widetilde K_N(s,t)=N^{-1/6}K_N(x_N(s),x_N(t))$. Since $dx=N^{-1/6}ds$, the kernel factors cancel and yield
\begin{equation}\label{eq:airy-mass}
 \int_{x_N(3)}^{x_N(4)}q_{N-1}(x)^2\,d\gamma(x)
 =\frac{\int_3^4\widetilde K_N(s,3)^2\,ds}
         {\widetilde K_N(3,3)}
 \xrightarrow[N\to\infty]{}
 \alpha:=\frac{\int_3^4K_{\mathrm{Ai}}(s,3)^2\,ds}
                  {K_{\mathrm{Ai}}(3,3)}>0.
\end{equation}
Indeed, $\operatorname{Ai}(s)>0$ for $s\ge0$, so the denominator and integrand in the limit are positive. Also $\alpha\le1$, since each $q_{N-1}$ has unit norm. Fix $\beta=\alpha/8>0$, independently of all subsequent degrees and dimensions. By convergence to $\alpha$, there is $N_0$ such that for every $N\ge N_0$ the left side of \eqref{eq:airy-mass} is at least $\alpha/2=4\beta$. The factor four leaves room for approximation errors in passing to the cube and comparing norms.

\subsection{Proof of \Cref{prop:airy-sharp}}
We can now transfer the construction to a Boolean subspace.
\begin{proof}
\emph{Step 1: reserving enough degree for the harmonic directions.}
For each sufficiently large integer $k$, set
\[
 R=\lfloor k^{1/3}\rfloor,\qquad K=k-R,\qquad N=K+1,\qquad
 a_k=2\sqrt k+k^{-1/6}.
\]
As $k\to\infty$, $N=k-\lfloor k^{1/3}\rfloor+1\to\infty$, so eventually $N\ge N_0$. We need the interval used above to lie beyond $a_k$, which is measured relative to the larger degree $k$.
For such $k$, $N\le k$ and
\[
 2(\sqrt k-\sqrt N)
 =\frac{2(R-1)}{\sqrt k+\sqrt N}\le2k^{-1/6},
\qquad 3N^{-1/6}\ge3k^{-1/6}.
\]
Hence $x_N(3)\ge a_k$, and \eqref{eq:airy-mass} implies that $\int q_K^2\1_{[a_k,\infty)}\,d\gamma\ge4\beta$.

\emph{Step 2: transfer to the cube with $k$ fixed.}
Fix one such $k$, and hence fix $R,K,N,a_k$ and the polynomial $q_K$. We now let the ambient dimension $d$ tend to infinity. For each $r\in\{0,\ldots,R\}$, the radial-weight identity
(Lemma~\ref{lem:harmonic-radial}) expresses energy in harmonic degree $r$ as an expectation over a cube with $d-2r$ coordinates. Accordingly, define the total energy and the energy above $a_k$ by
\[
 c_{r,d}=\E_{\mu_{d-2r}}[q_K(S_{d-2r}/\sqrt d)^2],
 \qquad
 b_{r,d}=\E_{\mu_{d-2r}}\!
 \left[q_K(S_{d-2r}/\sqrt d)^2
       \1_{\{S_{d-2r}/\sqrt d\ge a_k\}}\right].
\]

For each fixed $r$, the central limit theorem gives $S_{d-2r}/\sqrt d\rightarrow G$ in distribution, where $G\sim N(0,1)$. To obtain convergence of the expectations above, we must also control the growth of $q_K^2$. Since $q_K$ has fixed degree $K$ and
\[
 \sup_{d>2r}\E[|S_{d-2r}/\sqrt d|^{2K+2}]<\infty,
\]
the contribution to either expectation from $|S_{d-2r}/\sqrt d|>M$ tends to zero uniformly in $d$ as $M\to\infty$. Together with $\Pp(G=a_k)=0$, this gives
\[
 c_{r,d}\xrightarrow[d\to\infty]{}\E[q_K(G)^2]=1,
 \qquad
 b_{r,d}\xrightarrow[d\to\infty]{}
 \E[q_K(G)^2\1_{\{G\ge a_k\}}]\ge4\beta.
\]
These limits hold with $k$ fixed; no uniform convergence as $k$ grows is needed. Since there are only $R+1$ values of $r$, we may choose $d=d_k\ge k^3$ sufficiently large that
\[
 \frac12\le c_{r,d}\le2,
 \qquad b_{r,d}\ge2\beta
 \qquad\text{for every }r\in\{0,\ldots,R\}.
\]

\emph{Step 3: lift the one-dimensional witness to many directions.}
On the chosen cube, let $E=\{S_d/\sqrt d\ge a_k\}$ and
\[
 W_r=q_K(S_d/\sqrt d)\cH_{d,r},\qquad
 W=\bigoplus_{r=0}^{R}W_r.
\]
The degree of each product is at most $K+r\le k$. By \Cref{lem:harmonic-radial}, multiplication on $\cH_{d,r}$ changes the squared norm by the positive factor $c_{r,d}$, so it is injective. For a nonzero $h_r=fq_K(S_d/\sqrt d)\in W_r$, the same identity gives
\[
 \frac{\E[h_r^2\1_E]}{\E[h_r^2]}
 =\frac{b_{r,d}\|f\|_2^2}{c_{r,d}\|f\|_2^2}
 \ge\frac{2\beta}{2}=\beta.
\]
Different $W_r$ are orthogonal both globally and on $E$, by the same radial-weight identity. Thus
\begin{equation}\label{eq:airy-common}
 \dim W=\sum_{r=0}^{R}\left(\binom dr-\binom d{r-1}\right)=\binom dR,
 \qquad \E[h^2\1_E]\ge\beta \E[h^2]\quad(h\in W).
\end{equation}
This proves the common-event assertion for every linear combination.

\emph{Step 4: bound the common event's probability.}
The standard Chernoff argument uses $\E[e^{\lambda S_d/\sqrt d}]=\cosh(\lambda/\sqrt d)^d\le e^{\lambda^2/2}$.
Applying Markov’s inequality to \(e^{\lambda S_d/\sqrt d}\) with \(\lambda=a_k\) gives
\begin{equation*}
 \mu_d(E)\le e^{-a_k^2/2}
 =\exp\!\left(-2k-2k^{1/3}-\tfrac12k^{-1/3}\right)
 \le e^{-\Lambda_{d,k}-2k^{1/3}},
\end{equation*}
using $\Psi(q)\le2q$ from \Cref{lem:curve}.
The cross term in
$a_k^2/2=(2\sqrt k+k^{-1/6})^2/2$ is $2k^{1/3}$ and hence moving the cutoff above $2\sqrt k$ yields the required correction. Take $\rho_*=1-\beta$. Restricting $W$ to any $m$-dimensional subspace in \eqref{eq:airy-common} proves the lower inequality in \cref{prop:airy-sharp}. The upper inequality follows from the impossibility direction of \Cref{thm:common}, applied with leakage $\rho_*$ and, for example, $q_0=1/4$. Indeed, since $d_k\ge k^3$, we have $k/d_k\le k^{-2}\le1/4$ for all sufficiently large $k$.
\end{proof}

For each sufficiently large $k$, the transfer works once $d$ is large enough; choose one such dimension $d_k\ge k^3$. The argument does not give an explicit upper bound on $d_k$.

This proves that a positive correction of order $k^{1/3}$ is necessary at one fixed leakage level, for all subspace dimensions in the stated range. The smooth-window construction remains essential: it works for every $k\le q_0d$ and allows any prescribed leakage $\rho>0$. By contrast, the Airy construction guarantees only that a fixed fraction $\beta>0$ of the energy lies in $E$. The remaining fraction $1-\beta$ need not be small enough to establish the regression lower bound for a prescribed benchmark constant $A\ge16$.

\section*{AI methodology}
The author used OpenAI’s ChatGPT 5.6 Sol and 6.0 Astra for the tasks described below. ChatGPT 5.6 Sol assisted with literature research, proofreading and improving the clarity of exposition. In an extended exchange with ChatGPT 6.0 Astra, the author explored how to establish uniform sharpness of the uncertainty remainder. The model eventually suggested the Airy approach and provided much of the argument underlying the proof in Section~\ref{sec:airy}. The author independently re-derived and verified these proofs and takes full responsibility for the contents of the paper.

\phantomsection

\printbibheading[heading=bibintoc,title={References}]
\printbibliography[heading=none]

@misc{ASW15,
  author = {Emmanuel Abbe and Amir Shpilka and Avi Wigderson},
  title = {{Reed--Muller codes for random erasures and errors}},
  howpublished = {\emph{Proceedings of STOC}, pp.\ 297--306, 2015. Journal version: \emph{IEEE Transactions on Information Theory} 61(10):5229--5252, 2015.},
  note = {\href{https://arxiv.org/abs/1411.4590}{arXiv:1411.4590}.},
  shorthand = {ASW15},
  sortkey = {01}
}

@misc{Bec75,
  author = {William Beckner},
  title = {{Inequalities in Fourier analysis}},
  howpublished = {\emph{Annals of Mathematics} 102(1):159--182, 1975.},
  note = {\href{https://doi.org/10.2307/1970980}{doi:10.2307/1970980}.},
  shorthand = {Bec75},
  sortkey = {02}
}

@misc{BLM13,
  author = {St\'ephane Boucheron and G\'abor Lugosi and Pascal Massart},
  title = {{Concentration Inequalities: A Nonasymptotic Theory of Independence}},
  howpublished = {Oxford University Press, 2013.},
  note = {\href{https://doi.org/10.1093/acprof:oso/9780199535255.001.0001} {doi:10.1093/acprof:oso/9780199535255.001.0001}.},
  shorthand = {BLM13},
  sortkey = {04}
}

@misc{Bon70,
  author = {Aline Bonami},
  title = {{\'Etude des coefficients de Fourier des fonctions de $L^p(G)$}},
  howpublished = {\emph{Annales de l'Institut Fourier} 20(2):335--402, 1970.},
  note = {\href{https://doi.org/10.5802/aif.357}{doi:10.5802/aif.357}.},
  shorthand = {Bon70},
  sortkey = {05}
}

@misc{Bou02,
  author = {Olivier Bousquet},
  title = {{A Bennett concentration inequality and its application to suprema of empirical processes}},
  howpublished = {\emph{Comptes Rendus Math\'ematique} 334(6):495--500, 2002.},
  note = {\href{https://doi.org/10.1016/S1631-073X(02)02292-6}{doi:10.1016/S1631-073X(02)02292-6}.},
  shorthand = {Bou02},
  sortkey = {06}
}

@misc{CDL13,
  author = {Albert Cohen and Mark A. Davenport and Dany Leviatan},
  title = {{On the stability and accuracy of least squares approximations}},
  howpublished = {\emph{Foundations of Computational Mathematics} 13:819--834, 2013.},
  note = {\href{https://arxiv.org/abs/1111.4422}{arXiv:1111.4422}.},
  shorthand = {CDL13},
  sortkey = {07}
}

@misc{DG07,
  author = {Percy Deift and Dimitri Gioev},
  title = {{Universality at the edge of the spectrum for unitary, orthogonal, and symplectic ensembles of random matrices}},
  howpublished = {\emph{Communications on Pure and Applied Mathematics} 60(6):867--910, 2007.},
  note = {\href{https://doi.org/10.1002/cpa.20164}{doi:10.1002/cpa.20164}. \href{https://arxiv.org/abs/math-ph/0507023}{arXiv:math-ph/0507023}.},
  shorthand = {DG07},
  sortkey = {09}
}

@misc{EHE24,
  author = {El Hanchi, Ayoub and Chris J. Maddison and Murat A. Erdogdu},
  title = {{Minimax linear regression under the quantile risk}},
  howpublished = {\emph{Proceedings of COLT}, PMLR 247:1516--1572, 2024.},
  note = {\href{https://proceedings.mlr.press/v247/el-hanchi24a.html}{PMLR proceedings}.},
  shorthand = {EHE24},
  sortkey = {10}
}

@misc{EI22,
  author = {Alexandros Eskenazis and Paata Ivanisvili},
  title = {{Learning low-degree functions from a logarithmic number of random queries}},
  howpublished = {\emph{Proceedings of STOC}, pp.\ 203--207, 2022.},
  note = {\href{https://arxiv.org/abs/2109.10162}{arXiv:2109.10162}.},
  shorthand = {EI22},
  sortkey = {11}
}

@misc{EIS23,
  author = {Alexandros Eskenazis and Paata Ivanisvili and Lauritz Streck},
  title = {{Low-degree learning and the metric entropy of polynomials}},
  howpublished = {\emph{Discrete Analysis}, Article 17, 2023.},
  note = {\href{https://arxiv.org/abs/2203.09659}{arXiv:2203.09659}.},
  shorthand = {EIS23},
  sortkey = {12}
}

@misc{FM19,
  author = {Yuval Filmus and Elchanan Mossel},
  title = {{Harmonicity and invariance on slices of the Boolean cube}},
  howpublished = {\emph{Probability Theory and Related Fields} 175:721--782, 2019.},
  note = {\href{https://doi.org/10.1007/s00440-019-00900-w}{doi:10.1007/s00440-019-00900-w}. \href{https://arxiv.org/pdf/1507.02713v5}{arXiv:1507.02713v5}.},
  shorthand = {FM19},
  sortkey = {14}
}

@misc{KS21,
  author = {Naomi Kirshner and Alex Samorodnitsky},
  title = {{A moment ratio bound for polynomials and some extremal properties of Krawchouk polynomials and Hamming spheres}},
  howpublished = {\emph{IEEE Transactions on Information Theory} 67(6):3509--3541, 2021.},
  note = {\href{https://arxiv.org/pdf/1909.11929v1}{arXiv:1909.11929v1}.},
  shorthand = {KS21},
  sortkey = {16}
}

@misc{Men14,
  author = {Shahar Mendelson},
  title = {{Learning without concentration}},
  howpublished = {\emph{Proceedings of COLT}, PMLR 35:25--39, 2014.},
  note = {\href{https://proceedings.mlr.press/v35/mendelson14.html}{PMLR proceedings}.},
  shorthand = {Men14},
  sortkey = {22}
}

@misc{Men21,
  author = {Shahar Mendelson},
  title = {{Extending the scope of the small-ball method}},
  howpublished = {\emph{Studia Mathematica} 256(2):147--167, 2021.},
  note = {\href{https://doi.org/10.4064/sm190420-21-11}{doi:10.4064/sm190420-21-11}. \href{https://arxiv.org/abs/1709.00843}{arXiv:1709.00843}.},
  shorthand = {Men21},
  sortkey = {23}
}

@misc{Mou22,
  author = {Jaouad Mourtada},
  title = {{Exact minimax risk for linear least squares, and the lower tail of sample covariance matrices}},
  howpublished = {\emph{Annals of Statistics} 50(4):2157--2178, 2022.},
  note = {\href{https://arxiv.org/abs/1912.10754}{arXiv:1912.10754}.},
  shorthand = {Mou22},
  sortkey = {24}
}

@misc{Oli16,
  author = {Roberto Imbuzeiro Oliveira},
  title = {{The lower tail of random quadratic forms with applications to ordinary least squares}},
  howpublished = {\emph{Probability Theory and Related Fields} 166:1175--1194, 2016.},
  note = {\href{https://arxiv.org/abs/1312.2903}{arXiv:1312.2903}.},
  shorthand = {Oli16},
  sortkey = {25}
}

@misc{PS19,
  author = {Yury Polyanskiy and Alex Samorodnitsky},
  title = {{Improved log-Sobolev inequalities, hypercontractivity and uncertainty principle on the hypercube}},
  howpublished = {\emph{Journal of Functional Analysis} 277(11):108280, 2019.},
  note = {\href{https://arxiv.org/pdf/1606.07491v3}{arXiv:1606.07491v3}.},
  shorthand = {PS19},
  sortkey = {26}
}

@misc{Sam08,
  author = {Alex Samorodnitsky},
  title = {{A modified logarithmic Sobolev inequality for the Hamming cube and some applications}},
  howpublished = {\emph{arXiv preprint}, 2008.},
  note = {\href{https://arxiv.org/abs/0807.1679}{arXiv:0807.1679}.},
  shorthand = {Sam08},
  sortkey = {27}
}

@misc{Tropp12,
  author = {Joel A. Tropp},
  title = {{User-friendly tail bounds for sums of random matrices}},
  howpublished = {\emph{Foundations of Computational Mathematics} 12(4):389--434, 2012.},
  note = {\href{https://arxiv.org/pdf/1004.4389v7}{arXiv:1004.4389v7}.},
  shorthand = {Tropp12},
  sortkey = {30}
}

@misc{Yas14,
  author = {Pavel Yaskov},
  title = {{Lower bounds on the smallest eigenvalue of a sample covariance matrix}},
  howpublished = {\emph{Electronic Communications in Probability} 19, Article 83, pp.\ 1--10, 2014.},
  note = {\href{https://doi.org/10.1214/ECP.v19-3807}{doi:10.1214/ECP.v19-3807}. \href{https://arxiv.org/abs/1409.6188}{arXiv:1409.6188}.},
  shorthand = {Yas14},
  sortkey = {31}
}

@article{DS89,
  author  = {Donoho, David L. and Stark, Philip B.},
  title   = {Uncertainty Principles and Signal Recovery},
  journal = {SIAM Journal on Applied Mathematics},
  volume  = {49},
  number  = {3},
  pages   = {906--931},
  year    = {1989},
  doi     = {10.1137/0149053}
}

@article{LP62,
  author  = {Landau, H. J. and Pollak, H. O.},
  title   = {Prolate Spheroidal Wave Functions, Fourier Analysis
             and Uncertainty---{III}: The Dimension of the Space
             of Essentially Time- and Band-Limited Signals},
  journal = {Bell System Technical Journal},
  volume  = {41},
  number  = {4},
  pages   = {1295--1336},
  year    = {1962},
  doi     = {10.1002/j.1538-7305.1962.tb03279.x}
}

@inproceedings{HKZ12,
  author    = {Hsu, Daniel and Kakade, Sham M. and Zhang, Tong},
  title     = {Random Design Analysis of Ridge Regression},
  booktitle = {Proceedings of the 25th Annual Conference on Learning Theory},
  series    = {Proceedings of Machine Learning Research},
  volume    = {23},
  pages     = {9.1--9.24},
  year      = {2012},
  url       = {https://proceedings.mlr.press/v23/hsu12.html}
}

\clearpage
\appendix
\section{Supplementary proofs for the introduction}
\subsection{Minimax risk of the toy example}\label{app:toy-example}
Recall that we consider the data model
\[
Y_i=\theta h_{d,k}(X_i)+\varepsilon_i,
\qquad
\varepsilon_i\overset{\mathrm{iid}}{\sim}N(0,\sigma^2),
\] 
Let \(h=h_{d,k}\), \(p=\mathbb P(X\in E)\le2e^{-k/2}\), and \(\mathbb E[h^2\mathbf1_{E^c}]\le\rho\). Assume \(1\le n\le e^{k/2}/8\). We apply Le Cam's two-point method to any coefficients \(\theta,\theta+u\), choosing \(u^2=\sigma^2/(8n\rho)\). Write \(F=\{X_1,\ldots,X_n\notin E\}\) and \(Q_a=\mathbb P_a(\,\cdot\mid F)\). Then
\begin{align*}
\inf_{\widehat\theta}\sup_{a\in\mathbb R}
\mathbb P_a&\!\left\{|\widehat\theta-a|^2\ge
\frac{\sigma^2}{32n\rho}\right\}
\overset{\text{(i)}}{\ge} \mathbb P(F)\,
\inf_{\widehat\theta}\max_{a\in\{\theta,\theta+u\}}
Q_a\!\left\{|\widehat\theta-a|\ge u/2\right\}
\\
&\overset{\text{(ii)}}{\ge} \frac{\mathbb P(F)}2
\left(1-\sqrt{\frac{D_{\mathrm{KL}}(Q_\theta\|Q_{\theta+u})}{2}}\right)
\overset{\text{(iii)}}{=}\frac{\mathbb P(F)}2
\left(1-\sqrt{\frac{u^2}{4\sigma^2}
\mathbb E\!\left[\sum_{i=1}^n h(X_i)^2\mid F\right]}\right)
\\
&\overset{\text{(iv)}}{\ge}\frac38
\left(1-\sqrt{\frac{nu^2\rho}{4\sigma^2(1-p)}}\right)
\overset{\text{(v)}}{\ge}\frac38\left(1-\frac14\right)
=\frac9{32}.
\end{align*}
Here (i) restricts to two targets and to \(F\); (ii) is Le Cam’s two-point bound and Pinsker’s inequality; (iii) uses the Gaussian KL formula and its chain rule; (iv) uses \(\mathbb P(F)\ge1-np\ge3/4\) and \(\mathbb E[h(X)^2\mid X\notin E]\le\rho/(1-p)\); (v) substitutes \(u^2=\sigma^2/(8n\rho)\) and uses \(p\le1/4\). 

Thus, if \(\rho<1/(32A)\), error \(A\sigma^2/n\) cannot be guaranteed at confidence \(3/4\) throughout this exponential sample range.

\subsection{Exponential cost of noise}\label{app:noise-cost}

We prove a generalization of Corollary~\ref{cor:noise-cost} using harmonic polynomials of a fixed degree $R$.

\begin{proposition}\label{prop:same-model-general}
Fix $q_0\in(0,1/2)$ and $A\ge16$. There are constants $K_0,C_{A,q_0}>0$ such that, whenever
\[
 1\le k\le q_0d,\qquad
 0\le R\le k-K_0,\qquad
 1\le m\le\binom dR-\binom d{R-1},
\]
there exists an $m$-dimensional $V\subseteq\cP_{d,k}$ with the following properties for every $t\ge\log4$:
\begin{itemize}
\item $C2^R(m+t)$ noiseless observations suffice to recover every target in $V$ exactly, with probability at least
$1-e^{-t}$.
\item Under Gaussian noise of any variance $\sigma^2>0$,
\[
 \mathcal R^*_{n,e^{-t}}(V)
 >A\sigma^2\frac{m+t}{n}
 \quad\text{whenever}\quad
 1\le n\le
 (m+t)e^{\Lambda_{d,k}-C_{A,q_0}(R+k^{1/3})}.
\]
\end{itemize}
Here $C$ is a universal constant, and $V$ is independent of $t,\sigma$.
\end{proposition}

\begin{proof}
\emph{Step 1: constructing a concentrated radial factor.}
Choose $\rho_A\in(0,1)$ sufficiently small for \Cref{prop:rare-par}, and put $D=d-2R$, $K=k-R$. Set $M=M_{\rho_A/4,q_0}$ from \Cref{lem:scalar-witness}.
Since $K/D\le k/d\le q_0$, \Cref{lem:scalar-witness}, with leakage $\rho_A/4$, supplies a univariate polynomial $w$ of degree at most $K$ satisfying
\[
 \E_{\mu_D}[w(S_D)^2]=1,\qquad
 \E_{\mu_D}[w(S_D)^2\1_{\{S_D<u\}}]\le\rho_A/4,
\]
where we take
\[
 \tau_R=2\sqrt{(k-R)(d-k-R)},\qquad
 u=\tau_R-M\sqrt d\,K^{-1/6}.
\]
This cutoff is no larger than the lemma's cutoff. Recall the corresponding exponents
\[
 \mathcal E_{d,k}(R)
 =d\left[\log2-\HH\!\left(\frac12-\frac{\tau_R}{2d}\right)\right],
 \qquad \Lambda_{d,k}=d\Psi(k/d).
\]
The probability calculation in the proof of \Cref{prop:harmonic-lift}, together with $\mathcal E_{d,k}(R)\ge\Lambda_{d,k}-C_{q_0}R$, gives
\[
 0<p:=\mu_d(S_d\ge u)
 \le e^{-\mathcal E_{d,k}(R)+C_{A,q_0}K^{1/3}}
 \le e^{-\Lambda_{d,k}+C_{A,q_0}(R+k^{1/3})}.
\]
Choosing $K_0$ sufficiently large ensures $0<u<\tau_R<d$ and $p\le1/2$. For the latter, use $\mathcal E_{d,k}(R)\ge2(1-2q_0)K$.

Next, choose $0<\varepsilon<\sqrt{\rho_A}/4$ such that $w(s)+\varepsilon\ne0$ for every possible value $s\in\{-d,-d+2,\ldots,d\}$ of $S_d$, and define
\[
 g(s)=
 \frac{w(s)+\varepsilon}
 {\bigl(\E_{\mu_D}[(w(S_D)+\varepsilon)^2]\bigr)^{1/2}}.
\]
Then $g(S_d)$ never vanishes, and
\[
 \E_{\mu_D}[g(S_D)^2]=1,\qquad
 \E_{\mu_D}[g(S_D)^2\1_{\{S_D<u\}}]
 \le
 \left(\frac{\sqrt{\rho_A}/2+\varepsilon}
 {1-\varepsilon}\right)^2
 \le\rho_A.
\]
Non-zero $g(S_d)$ ensures that every noiseless response can be divided by the known factor $g(S_d(X_i))$, reducing recovery in $V$ to recovery in a lower-degree space $U$.

\emph{Step 2: constructing the model.}
Choose an $m$-dimensional $U\subseteq\cH_{d,R}$ and set
\[
 V=\{g(S_d)f:f\in U\},\qquad E=\{S_d\ge u\}.
\]
By \Cref{lem:harmonic-radial}, for every $f\in U$,
\[
 \E[(g(S_d)f)^2]=\E[f^2],\qquad
 \E[(g(S_d)f)^2\1_{E^c}]
 \le\rho_A\E[f^2].
\]
Thus $\dim V=m$, $V\subseteq\cP_{d,k}$, and, in an orthonormal basis, $\E[\Phi_V\Phi_V^\top\1_{E^c}]\preceq\rho_A I_m$.

\emph{Step 3: noiseless identification.}
Dividing each response by the known nonzero value $g(S_d(X_i))$ reduces recovery to $U$. Choose a basis $\psi_1,\ldots,\psi_m$ of $U$, and let
\[
 A_j=(\psi_\ell(X_i))_{\substack{1\le i\le j\\1\le\ell\le m}}
\]
be its evaluation matrix after $j$ observations. Full column rank means that these evaluations determine every coefficient uniquely.

Condition on $X_1,\ldots,X_j$ and suppose
$\operatorname{rank}(A_j)<m$.
Choose $0\ne a\in\ker A_j$ and set $f=\sum_{\ell=1}^m a_\ell\psi_\ell$.
Then $f$ is nonzero but vanishes at every previous input.
If $f(X_{j+1})\ne0$, the new evaluation row cannot lie in the span of the previous rows since those rows have zero inner product with $a$, whereas the new row does not.
Thus the rank increases.

Every nonzero degree-at-most-$R$ polynomial is nonzero on a fraction of at least $p_R:=2^{-R}$ of the cube.
Indeed, fixing the coordinates outside a monomial of maximal degree $r\le R$ preserves its nonzero coefficient, and hence each resulting restriction is nonzero on at least one of its $2^r$ inputs.
Since $X_{j+1}$ is independent of the previous inputs, the conditional probability of a rank increase is therefore at least $p_R$, until full rank is reached.
Reaching full rank within $n$ observations is therefore at least as likely as obtaining $m$ successes in $n$ independent trials with success probability $p_R$.
Thus
\[
 \Pp\{\operatorname{rank}(A_n)<m\}
 \le\Pp\{\operatorname{Bin}(n,p_R)<m\}.
\]
If $n\ge8p_R^{-1}(m+t)$, then $\mu:=np_R\ge8(m+t)$. The binomial lower-tail bound gives
\[
 \Pp\{\operatorname{Bin}(n,p_R)<m\}
 \le\Pp\{\operatorname{Bin}(n,p_R)\le\mu/2\}
 \le e^{-\mu/8}\le e^{-t}.
\]
On the full-rank event, every target is recovered exactly.

\emph{Step 4: noisy estimation.}
By \Cref{prop:rare-par},
\[
 \mathcal R^*_{n,e^{-t}}(V)
 >A\sigma^2\frac{m+t}{n}
 \qquad\text{for every positive integer}\quad
 n\le\frac{m+t}{1024p}.
\]
The bound on $p$ proves the claimed sample range after increasing $C_{A,q_0}$ sufficiently to absorb the factor $1024$.
\end{proof}

\begin{proof}[Proof of Corollary~\ref{cor:noise-cost}]
Apply \Cref{prop:same-model-general} with $R=1$.
Since $\dim\cH_{d,1}=d-1$, we can choose an $m$-dimensional subspace of $\cH_{d,1}$ for every $1\le m\le d-1$.
The noiseless bound is $O(m+t)$, and the noisy lower bound follows by absorbing the constant harmonic cost into $C_{A,q_0}k^{1/3}$.
The uniform OLS upper bound in \Cref{thm:main} gives the remaining assertion, again absorbing fixed factors into $C_{A,q_0}$.
\end{proof}

\section{The entropy curve and the cubic remainder}\label{app:moment}
We collect the elementary properties of the entropy curve and the
proof of \Cref{lem:cubic}, whose statement and use in clipping appear
in \Cref{sec:clipping}.

\subsection{The entropy curve}\label{app:curve}
Recall that all logarithms are natural and
\[
 \HH(x)=-x\log x-(1-x)\log(1-x),\qquad
 \Psi(q)=\log2-\HH\!\left(\frac12-\sqrt{q(1-q)}\right),
 \qquad \Lambda_{d,k}=d\Psi(k/d),
\]
with $0\log0=0$. The following properties follow by elementary
 differentiation; we include the details used later.

\begin{lemma}[Elementary calculus]\label[lemma]{lem:curve}
For $q\in[0,1/2]$, the function $\Psi$ is increasing and concave, with
\[
 \Psi(0)=0,\qquad\Psi(1/2)=\log2,\qquad\Psi'(0)=2.
\]
For any $q_0<1/2$, uniformly on $0\le q\le q_0$,
\begin{equation*}
 \Psi(q)=2q-\frac23q^2-\frac8{15}q^3+O_{q_0}(q^4),
 \qquad
 \frac{\Psi(q_0)}{q_0}q\le\Psi(q)\le2q.
\end{equation*}
\end{lemma}
\begin{proof}
\emph{Step 1: monotonicity and concavity.} Let $z=2\sqrt{q(1-q)}$ and $w=1-2q$, so $z^2+w^2=1$. Then
\begin{align*}
 \Psi(q)
 &=\frac{1+z}{2}\log(1+z)+\frac{1-z}{2}\log(1-z),\\
 \Psi'(q)&=\frac{2w}{z}\artanh z,\\
 \Psi''(q)&=\frac{4(z-\artanh z)}{z^3}<0.
\end{align*}
Here $\Psi'(q)>0$ and $\Psi''(q)<0$ for $0<q<1/2$.
The endpoint values and $\Psi'(0)=2$ follow by continuity.
Concavity places $\Psi$ above its chord from $0$ to $q_0$ and below
its tangent at $0$, proving the linear bounds.

\emph{Step 2: expansion at zero.} Integrating the power series for $\artanh z$ gives
\[
 \frac{1+z}{2}\log(1+z)+\frac{1-z}{2}\log(1-z)
 =\sum_{j\ge1}\frac{z^{2j}}{2j(2j-1)}.
\]
Substitute $z^2=4q(1-q)$ and collect the first three powers of $q$. The omitted term is uniformly $O_{q_0}(q^4)$ by analyticity at zero and smoothness on a compact interval ending below $1/2$.
\end{proof}

\subsection{Uniform cubic remainder}\label{app:moment-details}
\begin{proof}[Proof of \Cref{lem:cubic}]
\emph{Step 1: expressing the known rate as a minimum.}
This is an algebraic reformulation of the rate in \eqref{eq:KS-rate}.
Take $z=\tanh v$. Then
\begin{align*}
 \log\frac{(1+z)^p+(1-z)^p}{2}
 &=\log\cosh(pv)-p\log\cosh v,\\
 F(p,q)
 &=\min_{v>0}\{\log\cosh(pv)-p g_q(v)\},\\
 g_q(v)
 &=\log\cosh v+q\log\tanh v+\frac12\HH(q).
\end{align*}
Differentiating the above expression in braces gives
\begin{equation}\label{eq:variational-sign}
 \begin{aligned}
 \frac{\partial}{\partial v}
 \{\log\cosh(pv)-p g_q(v)\}
 &=\frac{p}{\sinh v\cosh v}\bigl(q_p(v)-q\bigr),\\
 q_p(v)
 =\sinh v\cosh v\bigl(\tanh(pv)-&\tanh v\bigr)
 =\frac12\left(1-
 \frac{\cosh((p-2)v)}{\cosh(pv)}\right).
 \end{aligned}
\end{equation}
For $p\ge2$, the ratio on the last line decreases strictly from one to zero. This can be seen by inspecting its logarithmic derivative
\(
 (p-2)\tanh((p-2)v)-p\tanh(pv)<0.
\)

Thus $q_p$ increases strictly from zero to $1/2$. The derivative in \eqref{eq:variational-sign} changes sign exactly once, from negative to positive. This proves the variational form and uniqueness. It also shows that evaluation at any value of $v$ gives an upper bound on $F$.

\emph{Step 2: choosing a test point $v$.}
Fix $q\in(0,q_0]$ and set
\begin{equation}\label{eq:a-param}
 a=\artanh\sqrt{\frac q{1-q}},
 \qquad
 q=\frac{\sinh^2a}{\cosh(2a)}.
\end{equation}
At $p=2$, the minimizer is $v=a$. For $p=2+2s$, we can use $v=a/(1+s)$ since this keeps $pv=2a$ fixed, making the first term of the
variational expression constant. Define
\[
 M_q(s)=\log\cosh(2a)-2(1+s)g_q\!\left(\frac a{1+s}\right).
\]
Then $F(2+2s,q)\le M_q(s)$.

\emph{Step 3: verifying cancellation up to the second order.}
Let $a_0=\artanh\sqrt{q_0/(1-q_0)}$. Basic algebra using \eqref{eq:a-param} yields
\begin{equation}\label{eq:g-identities}
 \begin{aligned}
 2g_q(a)&=\log\cosh(2a),\\
 g_q'(a)&=\tanh(2a),\\
 g_q''(a)&=0,\\
 2a\tanh(2a)-\log\cosh(2a)&=\Psi(q).
 \end{aligned}
\end{equation}
To verify these identities, one can simply take $z=\tanh a$. Then $q=z^2/(1+z^2)$ and
\begin{align*}
 \HH(q)&=\log(1+z^2)-2q\log z,\\
 g_q(a)&=\log\cosh a+\tfrac12\log(1+z^2)
          =\tfrac12\log\cosh(2a),\\
 g_q''(a)
 &=\operatorname{sech}^2a
 -q\frac{\cosh(2a)}{\sinh^2a\cosh^2a}=0.
\end{align*}
The first derivative follows from
$g_q'(a)=\tanh a+q/(\sinh a\cosh a)=\tanh(2a)$.
Also $\tanh(2a)=2\sqrt{q(1-q)}$. For $u=\tanh(2a)$,
\(
 \log2-\HH((1-u)/2)=2au-\log\cosh(2a),
\)
which proves the last line of \eqref{eq:g-identities}.

Write $w=1+s$ and $v=a/w$. Differentiation yields
\begin{equation}\label{eq:M-derivatives}
 \begin{aligned}
 M_q'(s)&=2\{vg_q'(v)-g_q(v)\},\\
 M_q''(s)&=-\frac{2v^2}{w}g_q''(v),\\
 M_q'''(s)&=\frac2{w^2}
       \{3v^2g_q''(v)+v^3g_q'''(v)\}.
 \end{aligned}
\end{equation}
It follows from \eqref{eq:g-identities} that
$M_q(0)=0$, $M_q'(0)=\Psi(q)$, and $M_q''(0)=0$.

\emph{Step 4: bounding the remainder uniformly as $q\downarrow0$.}
Taylor's theorem concludes the proof once $|M_q'''(s)|\le C_{q_0}q$. 

Let $\ell(v)=\log\tanh v$. Its derivatives satisfy
\begin{align*}
 \ell''(v)&=-\frac{4\cosh(2v)}{\sinh^2(2v)},\\
 \ell'''(v)&=-\frac8{\sinh(2v)}
             +\frac{16\cosh^2(2v)}{\sinh^3(2v)}.
\end{align*}
For $0<v\le a_0$, the quantities $v^2|\ell''(v)|$ and $v^3|\ell'''(v)|$ are bounded by constants depending only on $a_0$. This can be seen directly from $\sinh(2v)\ge2v$, since their limits at zero are $1$ and $2$. Moreover,
\[
 v^2\le a^2
 \le q\cosh(2a_0),
\]
because $q=\sinh^2a/\cosh(2a)$ and $\sinh a\ge a$. The derivatives of $\log\cosh v$ are bounded on $[0,a_0]$.
In $g_q(v)=\log\cosh v+q\ell(v)+\HH(q)/2$, the first term therefore contributes $O_{q_0}(v^2)$ and $O_{q_0}(v^3)$ to the weighted second and third derivatives respectively, while the $\ell$ term contributes $O_{q_0}(q)$. Since $v\le a_0$ and $v^2=O_{q_0}(q)$, this yields
\[
 |v^2g_q''(v)|\le C_{q_0}q~, \quad|v^3g_q'''(v)|\le C_{q_0}q.
\]
For $0\le s\le1/2$, \eqref{eq:M-derivatives} yields
$|M_q'''(s)|\le C_{q_0}q$.
Taylor’s theorem and the bound on the third derivative give
$M_q(s)\le s\Psi(q)+C_0qs^3$.
Since $F(2+2s,q)\le M_q(s)$, this proves \eqref{eq:cubic-rate}. Finally, we apply \Cref{thm:KS} and multiply by $d$, using $dq=k$, to obtain \eqref{eq:cubic-moment}.
\end{proof}

\section{Classical analytic and probabilistic tools}\label{app:tools}\label{app:empirical}
This section collects the hypercontractive, empirical-process, and Gaussian probability bounds used in the sampling arguments.

\subsection{Bonami--Beckner hypercontractivity}
We use the classical degree form of the Bonami--Beckner inequality
\cite{Bon70,Bec75}.
\begin{lemma}[Bonami--Beckner inequality]\label[lemma]{lem:bonami-beckner}
For $p\ge2$ and every $h\in\cP_{d,k}$,
\[
 \bigl(\E[|h|^p]\bigr)^{1/p}
 \le (p-1)^{k/2}\bigl(\E[h^2]\bigr)^{1/2}.
\]
In particular, for $\E[h^2]=1$,
\[
 \E[h^4]\le9^k,\qquad
 \log\E[|h|^{2+2s}]\le k(1+s)\log(1+2s)
 \quad(s\ge0).
\]
\end{lemma}

\subsection{A bound independent of the ambient dimension}
For our unrestricted-dimension result we also need the classical Bonami--Beckner hypercontractive estimate. The clipping step is the same as in \Cref{prop:Tupper}.
\begin{proposition}[Classical hypercontractive comparison]\label[proposition]{prop:classical-upper}
For $k\ge1$, $\eta\in(0,1)$, and $T_k(\eta)=\sup_{d\ge1}T_{d,k}(\eta)$,
\[
 \log T_k(\eta)\le 2k+
 \left(\frac16+2\log\frac1{1-\eta}\right)k^{1/3}.
\]
\end{proposition}
\begin{proof}
If $\E[h^2]=1$ and $\deg h\le k$, \Cref{lem:bonami-beckner} with $p=2+2s$ gives
\[
 \log \E[|h|^{2+2s}]\le k(1+s)\log(1+2s)
 \le 2ks+\frac23ks^3\qquad(0\le s\le1/2).
\]
\emph{Step 1: a Taylor bound.} Indeed, $f(s)=(1+s)\log(1+2s)$ satisfies $f(0)=f''(0)=0$, $f'(0)=2$, and $f'''(s)=4(1-2s)/(1+2s)^3\le4$ on this interval. 

\emph{Step 2: clipping.} We take $L=\log(1/(1-\eta))$, $s=k^{-1/3}/2$, and $\log B=2k+(2/3)ks^2+L/s$. Then
\[
 \E[(h^2-B)_+]\le B^{-s}\E[|h|^{2+2s}]\le e^{-L}=1-\eta,
\]
so $\E[h^2\wedge B]\ge\eta$. Substituting $s$ proves the bound uniformly over $d$ and $h$.
\end{proof}

\subsection{A bounded empirical-process inequality}
We use the following result in its bounded, centered form. For convenience, we include the zero function in the class,
ensuring that $Z\ge0$.
\begin{theorem}[Bousquet, Theorem 2.3 of \cite{Bou02}]\label[theorem]{thm:bousquet}
Let $X_1,\ldots,X_n$ be independent and identically distributed, and let $\mathcal G$ be a countable class of measurable real functions such that
\[
 \E[g]=0,\qquad |g|\le1,\qquad \E[g^2]\le v\quad(g\in\mathcal G).
\]
Assume $0\in\mathcal G$, and put $Z=\sup_{g\in\mathcal G}\sum_{i=1}^ng(X_i)$. Then, for every $t>0$,
\[
 \Pp\left(
 Z>\E Z+\sqrt{2(nv+2\E Z)t}+\frac t3
 \right)\le e^{-t}.
\]
\end{theorem}
By continuity in the coefficient vector, all suprema may be taken
over a countable dense subset of the unit sphere, so Bousquet's
countability assumption is satisfied.

\subsection{Symmetrization and contraction}
The following are two standard results, see
\cite[Lemma~11.4 and Theorem~11.6]{BLM13}.

\begin{lemma}[Symmetrization and contraction]
\label{lem:symmetrization-contraction}
Let $\varepsilon_1,\ldots,\varepsilon_n$ be independent random
signs, each uniform on $\{-1,1\}$.

\emph{(i) Symmetrization.}
Let $X_1,\ldots,X_n$ be independent and identically distributed,
independently of the signs. For a countable class $\mathcal F$
of integrable functions,
\[
 \E\sup_{f\in\mathcal F}
 \left|\E_n[f]-\E[f]\right|
 \le
 2\E\sup_{f\in\mathcal F}
 \left|\frac1n\sum_{i=1}^n\varepsilon_i f(X_i)\right|.
\]
The expectation on the right includes both the sample and the signs.

\emph{(ii) Contraction.}
For a bounded set $\mathcal A\subseteq\R^n$ and an $L$-Lipschitz
function $\varphi:\R\to\R$ satisfying $\varphi(0)=0$,
\[
 \E_\varepsilon\sup_{a\in\mathcal A}
 \left|\sum_{i=1}^n\varepsilon_i\varphi(a_i)\right|
 \le
 2L\E_\varepsilon\sup_{a\in\mathcal A}
 \left|\sum_{i=1}^n\varepsilon_i a_i\right|.
\]
\end{lemma}

\subsection{Elementary Gaussian probability bounds}
\begin{lemma}[Gaussian and chi-square tail bounds]\label[lemma]{lem:gaussian-bounds}
For $r\ge1$, the chi-squared variable defined in \Cref{sec:model} satisfies the following bounds.
For $u\ge0$ and $0<a<1$,
\begin{align}
 \Pp(\chi_r^2\ge4(r+u))&\le e^{-u},\label{eq:chi-upper}\\
 \Pp(\chi_r^2\le ar)&\le(ae^{1-a})^{r/2}.
 \label{eq:chi-lower}
\end{align}
For a one-dimensional standard Gaussian $G$ and every $t\ge\log4$,
\begin{equation}\label{eq:strict-tail}
 \frac{15}{16}e^{-t/256}\Pp(G^2\ge t/64)>e^{-t}.
\end{equation}
\end{lemma}
\begin{proof}
Completing the square in the Gaussian integral gives
$\E[e^{v\chi_r^2}]=(1-2v)^{-r/2}$ for $v<1/2$.
Exponential Markov with $v=1/4$ gives
\begin{align*}
 \Pp(\chi_r^2\ge4(r+u))
 &\le e^{-r-u}2^{r/2}\le e^{-u}.
\end{align*}
For the lower tail use $v=-(1-a)/(2a)<0$:
\[
 \Pp(\chi_r^2\le ar)
 \le e^{-var}(1-2v)^{-r/2}=(ae^{1-a})^{r/2}.
\]
To prove \eqref{eq:strict-tail}, first suppose $\log4\le t\le16$. Then $\sqrt{t/64}\le1/2$. The Gaussian density is bounded by $(2\pi)^{-1/2}$, so
\[
 \Pp(|G|\ge1/2)\ge1-(2\pi)^{-1/2}>3/5.
\]
The left side of \eqref{eq:strict-tail} is then greater than
$(15/16)e^{-1/16}(3/5)>1/2$, while $e^{-t}\le1/4$.

For $t\ge16$, put $b=\sqrt t/8$. Integrate over $[b,b+1]$ and $[-b-1,-b]$:
\[
 \Pp(G^2\ge t/64)
 \ge\sqrt{2/\pi}\,e^{-(b+1)^2/2}
 =\sqrt{2/\pi}\,e^{-1/2}
     e^{-t/128-\sqrt t/8}
 >\frac14e^{-5t/128},
\]
since $\sqrt t/8\le t/32$. Thus the left side of \eqref{eq:strict-tail} exceeds
$(15/64)e^{-11t/256}$. This is larger than $e^{-t}$ for every $t\ge16$.
\end{proof}

\begin{corollary}[OLS after lower conditioning]\label[corollary]{cor:ols}
If $\Pp(\hSigma_V\succeq\tfrac12I_m)\ge1-\delta/2$, where $\delta\le1/4$, then
\[
 \mathcal R^*_{n,\delta}(V)
 \le12\sigma^2\frac{m+\log(1/\delta)}n.
\]
OLS satisfies the same bound uniformly over $\theta$.
\end{corollary}
\begin{proof}
Conditionally on a design with $\hSigma_V\succeq I_m/2$,
\[
 \|\widehat\theta_{\rm OLS}-\theta\|_2^2
 \ \preceq_{\rm st}\ \frac{2\sigma^2}{n}\chi_m^2.
\]
Use \eqref{eq:chi-upper} with $u=t+\log2$, where $t=\log(1/\delta)$, and add the design failure probability. Since $t\ge\log4$, $\log2\le t/2$, and
\(
 8(m+t+\log2)\le12(m+t).
\)

The minimax assertion follows because OLS is an admissible estimator (or from \Cref{thm:exact-minimax}).
\end{proof}

\section{The radial-weight identity}\label{app:harmonic-details}
We derive \Cref{lem:harmonic-radial} directly from
\cite[Theorem~3.10]{FM19}. We first apply that theorem to the uniform cube to identify its measure-independent constant, and then to the
cube weighted by $H(S_d)$.
\begin{proof}
Filmus--Mossel's theorem states that, for homogeneous harmonic multilinear $f,g$ of degrees $r,s$ and any exchangeable probability measure $\alpha$, $\E_\alpha[fg]=0$ if $r\ne s$. If $r=s$, there is a constant $C_{f,g}$ independent of $\alpha$ such that
\[
 \E_\alpha[fg]
 =C_{f,g}\,\E_\alpha\left[
       \prod_{j=1}^r(x_{2j-1}-x_{2j})^2\right].
\]
Apply this first with $\alpha=\mu_d$. Each squared difference has expectation $2$, and the pairs are independent, hence
\[
 C_{f,g}=2^{-r}\ip{f}{g}_d.
\]
For $H\ge0$ with $Z=\E_{\mu_d}[H(S_d)]>0$, now choose
\[
 \alpha(x)=\frac{H(S_d(x))}{Z}\,\mu_d(x).
\]
This probability measure is exchangeable because both $S_d$ and $\mu_d$ are invariant under coordinate permutations. Substituting it into the theorem and multiplying by $Z$ gives
\[
 \E_{\mu_d}[fgH(S_d)]
 =\ip{f}{g}_d\,\E_{\mu_d}[Q_r^2H(S_d)],
 \qquad
 Q_r=\prod_{j=1}^r\frac{x_{2j-1}-x_{2j}}{\sqrt2}.
\]
On the other hand, for unequal degrees, the same substitution yields zero.

To evaluate the remaining expectation, let $A_r$ be the event that the signs differ in each of the first $r$ pairs. Then
\[
 \Pr_{\mu_d}(A_r)=2^{-r},\qquad Q_r^2=2^r\1_{A_r}.
\]
Conditional on $A_r$, those pairs contribute zero to $S_d$, while the remaining $d-2r$ coordinates are independent uniform signs. Thus
\[
 \E_{\mu_d}[Q_r^2H(S_d)]
 =\E_{\mu_d}[H(S_d)\mid A_r]
 =\E_{\mu_{d-2r}}[H(S_{d-2r})],
\]
which is the equal-degree case of \eqref{eq:harmonic-radial}. The empty product for $r=0$ is one. If $H\ge0$ and $Z=0$, the weighted identities are immediate since $H(S_d)=0$ on the cube. Finally, decomposing a real $H$ into its positive and negative parts extends
the result to signed weights by linearity.
\end{proof}

\section{Probability constants for the sampling lower bounds}\label{app:lower}
We give the complete probability bounds for the Gaussian regression lower bounds derived from \Cref{lem:rare-cov}. Throughout, $\Pp(E)=p\in(0,1/2]$.
We derive sampling lower bounds under either of two concentration assumptions: every function in $V$ has at most a fraction $\rho$ of its energy outside $E$, or, when $t\ge m$, one unit function in $V$ has this property. These are the two assumptions of \Cref{lem:rare-cov}.

\begin{lemma}[Baseline sample complexity]\label[lemma]{lem:baseline}
For $d,k\ge1$, $1\le m\le D_{d,k}$, $0<\delta\le1/4$, and $A\ge16$, there is a universal $c_0>0$ such that
\[
 \Npar(d,k,m,A,\delta)
 \ge c_0\bigl(m+\log(1/\delta)\bigr).
\]
\end{lemma}
\begin{proof}
If $n<m$, every feature matrix is singular, so the regression guarantee fails. This gives the lower bound $m$.

For the confidence term, include $h(x)=(1+x_1)/\sqrt2$ in the model and extend it to an orthonormal basis of any feasible size $m$. If $t=\log(1/\delta)$, then
\[
 \begin{aligned}
 n\le\frac{t}{2\log2}
 &\quad\Longrightarrow\quad
 \Pp(X_{i1}=-1\text{ for every }i)
 =2^{-n}\ge e^{-t/2}>\delta.
 \end{aligned}
\]
On this event the feature column for $h$ is zero, so once again the regression guarantee fails. Taking an integer sample size below the displayed bound yields a universal constant times $t$. If that range contains no positive integer, use $N\ge1$ . Finally, use $m\vee t\ge(m+t)/2$.
\end{proof}

\begin{proposition}[Estimator-independent parametric obstruction]\label[proposition]{prop:rare-par}
Fix $A\ge16$. There is $\rho_A>0$ such that, under $\E[\Phi(X)\Phi(X)^\top\1_{\{X\notin E\}}]\preceq\rho I_m$, or under the directional assumption in \Cref{lem:rare-cov}(ii) when $t\ge m$, with $\rho\le\rho_A$, for every $t\ge\log4$ and every positive integer
\(
 n\le(m+t)/(1024p),
\)
the Gaussian model satisfies
\[
 \mathcal R^*_{n,e^{-t}}(V)
 >A\sigma^2\frac{m+t}{n}.
\]
\end{proposition}
\begin{proof}
Take $b_0=1/(1024e)$ and choose
\(
 0<\rho_A<
 \min\left\{\frac18,\frac{b_0}{3072A},
                         \frac1{16384A}\right\}.
\)
Use the random variable $Z_{n,V}$ in \Cref{thm:exact-minimax}. A singular design makes it infinite, so it is enough to focus on nonsingular designs.

Suppose first that $t<m$. Then $m\ge2$. With probability at least $15/16$, at least $r=\lfloor m/2\rfloor\ge m/3$ eigenvalues of $\hSigma$ are at most $128\rho$. Conditional on any such nonsingular design,
\[
 Z_{n,V}\ \succeq_{\rm st}\
 \frac{\sigma^2}{128\rho n}\chi_r^2.
\]
By \eqref{eq:chi-lower},
\(
 \Pp(\chi_r^2<b_0r)\le(e b_0)^{r/2}\le1/32.
\)
Thus, with unconditional probability at least $(15/16)(31/32)>1/4$,
\[
 Z_{n,V}
 \ge\frac{\sigma^2b_0r}{128\rho n}
 \ge\frac{\sigma^2b_0m}{384\rho n}
 >8A\sigma^2\frac mn
 >A\sigma^2\frac{m+t}{n}.
\]

If $t\ge m$, \Cref{lem:rare-cov}(ii) gives an event $D$ of probability at least $(15/16)e^{-t/256}$ on which $\lambda_1(\hSigma)\le32\rho$. Conditional on such a design, one coordinate of the independent standard Gaussian vector yields
\[
 Z_{n,V}\ \succeq_{\rm st}\
 \frac{\sigma^2}{32\rho n}G^2.
\]
By \eqref{eq:strict-tail}, the event
\[
 Z_{n,V}\ge\frac{\sigma^2t}{2048\rho n}
\]
has probability strictly greater than $e^{-t}$. Its displayed threshold is greater than $8A\sigma^2t/n$, hence greater than the benchmark rate because $m\le t$. This proves that the lower $(1-e^{-t})$-quantile of $Z_{n,V}$ strictly exceeds the benchmark rate. Lastly, apply \Cref{thm:exact-minimax}.
\end{proof}

\section{Comparison to the Polyanskiy--Samorodnitsky uncertainty boundary}\label{app:baselines}
We restate the Polyanskiy--Samorodnitsky uncertainty principle in our notation and explain how it yields the leading exponent of our one-dimensional clipping threshold when $k/d$ tends to a fixed value in $(0,1/2)$.

Let $B_{d,r}$ denote a Hamming ball of radius $\lfloor rd\rfloor$ and put $r_*(q)=\tfrac12-\sqrt{q(1-q)}$.

Let $\Pi_{\le k}$ be orthogonal projection onto $\cP_{d,k}$ and $C_E=\Pi_{\le k}\1_E\Pi_{\le k}$ its concentration operator, where $\1_E$ acts by multiplication. Its eigenvalues are denoted $\kappa_1(C_E)\ge\kappa_2(C_E)\ge\cdots$.

Then, Theorem~9 of \cite{PS19} implies the following statements for fixed $q\in(0,1/2)$ and fixed $r\in(0,1/2)$, with $k=\lfloor qd\rfloor$: for some $c=c(q,r)>0$ and all sufficiently large $d$,
\[
 \begin{aligned}
 r<r_*(q),\quad |E|\le e^{d\HH(r)}
 &\quad\Longrightarrow\quad \kappa_1(C_E)\le e^{-cd},\\
 r>r_*(q),\quad E=B_{d,r}
 &\quad\Longrightarrow\quad \kappa_1(C_E)\ge1-e^{-cd}.
 \end{aligned}
\]
The cosine of their angle between the spatial and Fourier subspaces is the square root of $\kappa_1(C_E)$; squaring only changes the constants in these bounds. The boundary equation is $(1-2r)^2+(1-2q)^2=1$, and $\mu_d(B_{d,r})=\exp\{-d(\log2-\HH(r))+o(d)\}$. Thus the critical probability exponent is precisely $\Lambda_{d,k}$.

For each fixed $\eta\in(0,1)$, these statements imply
\[
 \log T_{d,\lfloor qd\rfloor}(\eta)=d\Psi(q)+o(d).
\]

For the upper bound, choose $B=\exp\{d(\Psi(q)+\varepsilon)\}$ with fixed $\varepsilon>0$. For every unit polynomial $h$ of degree at most $\lfloor qd\rfloor$, the set $E=\{h^2>B\}$ has probability at most $1/B$. The first implication above therefore gives $\E[h^2\1_E]=o(1)$, uniformly in $h$, and hence $\E[h^2\wedge B]\ge1-o(1)\ge\eta$ for sufficiently large $d$.

For the lower bound, the second implication provides a unit polynomial $h$ with $\E[h^2\1_{E^c}]=o(1)$ on a Hamming ball $E$ whose probability exponent is arbitrarily close to $\Psi(q)$. Choosing this ball so that $\mu_d(E)\le\exp\{-d(\Psi(q)-\varepsilon/2)\}$ and setting $B=\exp\{d(\Psi(q)-\varepsilon)\}$ gives
\[
 \E[h^2\wedge B]
 \le \E[h^2\1_{E^c}]+B\mu_d(E)
 =o(1)<\eta.
\]
Thus $\log T_{d,\lfloor qd\rfloor}(\eta)/d\to\Psi(q)$.

\section{Further bounds and model examples}\label{app:consequences}
We derive the additional bounds and justify the model examples discussed in \Cref{sec:consequences}.

\subsection{Ambient dimension and model dimension}\label{app:modeldimscaling}

\paragraph{Worst case over the ambient dimension.}
Let $\Npar^\infty$ be the least threshold valid on every cube admitting a model of dimension $m$. We show that
\[
 \log\frac{\Npar^\infty}{m+t}=2k+O(k^{1/3}),
\]
with constants depending on the fixed target parameters.

For the upper bound, \Cref{prop:classical-upper,prop:clip-cov,cor:ols} give
\[
 \Npar^\infty\le C(m+t)\exp\{2k+Ck^{1/3}\},
\]
uniformly over $d$, including $k>d$. For the lower bound, it suffices to exhibit one suitable ambient dimension. Choose
\[
 d=\left\lceil\max\{m+1,10k,k^{5/3}\}\right\rceil.
\]
Then $m<d\le\Mhard$ and $k/d\le1/10$, so the matching lower bound applies. Moreover,
\[
 \Lambda_{d,k}\ge2k-Ck^2/d\ge2k-Ck^{1/3}.
\]
The lower bound of \Cref{thm:main} therefore proves the claim.
The same choice of $d$, together with the classical upper bound and \Cref{prop:clipping}, also gives
\[
 \log\sup_d T_{d,k}(\eta)=2k+O_\eta(k^{1/3}).
\]

\paragraph{The finite-dimensional correction.}

By \Cref{lem:curve},
\[
 \Psi(q)=2q-\frac23q^2-\frac8{15}q^3+O(q^4).
\]
Substituting this expansion into \Cref{thm:main} gives, throughout the range $m\le M_{d,k}\vee t$ and as $k/d\to0$,
\[
 \log\frac{\Npar}{m+t}
 =2k-\frac{2k^2}{3d}
   +O\!\left(\frac{k^3}{d^2}+k^{1/3}\right).
\]
Thus the first correction to the unrestricted-dimension exponent $2k$ is $-2k^2/(3d)$. Relative to its magnitude, the two error terms have orders $k/d$ and $d/k^{5/3}$, respectively.
Both vanish when $k\ll d\ll k^{5/3}$. Alternatively, if $k/d\to q\in(0,1/2)$, then
\[
 \frac1d\log\frac{\Npar}{m+t}\longrightarrow\Psi(q)
\]
throughout the matching range.
When $k\ll d\ll k^{5/3}$ and $m\le\Mhard\vee t$, the threshold is smaller than its unrestricted-dimension counterpart by a factor
\[
 \frac{\Npar^\infty}
      {\Npar}
 =\exp\!\left\{\left(\frac23+o(1)\right)\frac{k^2}{d}\right\}.
\]

\paragraph{A lower bound for larger model dimensions.}
Fix $q_0\in(0,1/2)$ and $A\ge16$, and assume $1\le k\le q_0d$ and $t=\log(1/\delta)\ge\log4$.
By allocating more degree to the harmonic directions, we can extend the lower bound beyond $m\le\Mhard$, at the cost of a larger error. Specifically, for any integer $0\le R\le k$ with $m\le\binom dR$,
\[
 \log\frac{\Npar}{m+t}
 \ge \Lambda_{d,k}-C_{q_0,A}(R+k^{1/3}).
\]

To prove this, \Cref{prop:harmonic-lift} supplies a common concentration set of probability at most $e^{-a_R}$, where
\[
 a_R=\mathcal E_{d,k}(R)-C(k-R)^{1/3}.
\]
If $a_R\ge\log2$, this probability is at most $1/2$, so
\Cref{prop:rare-par} gives
\[
 \log\frac{\Npar}{m+t}\ge a_R-O(1)
 \ge \Lambda_{d,k}-C'(R+k^{1/3}),
\]
where the last inequality uses \eqref{eq:harmonic-cost}. If $a_R<\log2$, the same entropy comparison gives $\Lambda_{d,k}\le C'(R+k^{1/3})+O(1)$. In this case, the baseline bound $\Npar\ge c_0(m+t)$ from \Cref{lem:baseline} proves the claim after increasing the constant.

\paragraph{The leading exponent for subexponential model dimensions.}
Suppose $k/d\to q\in(0,1/2)$ and $\log m=o(d)$. Then, even outside the cube-root matching range,
\[
 \log\frac{\Npar}{m+t}=\Lambda_{d,k}+o(d),
\]
where $\Lambda_{d,k}=\Theta(d)=\Theta(k)$.

Indeed, choose $R=\lceil\log_2m\rceil$, with $R=0$ when $m=1$. We have $R=o(d)$, so eventually $R\le k/2$ and $d\ge2R$. For $R>0$,
\[
 \binom dR\ge(d/R)^R\ge2^R\ge m.
\]
The preceding lower bound therefore applies, with error $O(R+k^{1/3})=o(d)$. The upper bound of \Cref{thm:main} also has error $o(d)$ and holds for every feasible $m$. Together these establish the leading exponent, uniformly in $t\ge\log4$.

\subsection{Known Walsh coordinates and finite-space caps}\label{app:cap}
For an isotropic feature vector with $\|\Phi(x)\|_2^2\le L$, matrix Chernoff~\cite[Corollary 5.2]{Tropp12} gives, for $0<c<1$,
\begin{equation}\label{eq:finite-cap}
 \Pp(\lambda_{\min}(\hSigma)\le c)
 \le m\exp\!\left(-\frac{(1-c)^2n}{2L}\right).
\end{equation}
Indeed, its matrix summands are positive semidefinite, have norm at most $L$, and their expectations sum to $nI_m$. The standard Chernoff factor is $[e^{-u}/(1-u)^{1-u}]^{n/L}$ with $u=1-c$. The inequality $-u-(1-u)\log(1-u)\le-u^2/2$ follows by differentiation.

For a known span of $m$ Walsh characters, $\|\Phi(x)\|_2^2=m$ exactly. Hence $n\ge C_cm(\log m+t)$ gives stable sampling. Choosing $c=1/2$, failure probability $\delta/2$, and using \Cref{cor:ols} gives parametric prediction. Character values need not be mutually independent; only independent sample rows are used.

Let $\chi_{\le k}(x)=(\chi_S(x))_{|S|\le\min(k,d)} \in\R^{D_{d,k}}$ be the column vector of Walsh characters. Expanding each basis function of $V$ in this basis gives
\[
 \Phi_V(x)=U\chi_{\le k}(x),
\]
where $U\in\R^{m\times D_{d,k}}$ contains the expansion coefficients.

Since $U$ has orthonormal rows, $\|\Phi_V(x)\|_2^2\le\|\chi_{\le k}(x)\|_2^2=D_{d,k}$.
Thus \eqref{eq:finite-cap} gives
\[
 N\le C D_{d,k}(\log m+t)
\]
for either the stable-sampling or parametric threshold, with $C$ depending on the fixed target parameters.

Likewise, every unit polynomial $h\in\cP_{d,k}$ satisfies $h(x)^2\le D_{d,k}$. The inequality $z\wedge(\eta D_{d,k})\ge\eta z$ for $0\le z\le D_{d,k}$ therefore gives
\[
 \E[h^2\wedge(\eta D_{d,k})]\ge\eta,
 \qquad T_{d,k}(\eta)\le\eta D_{d,k}.
\]
When $k\ge d$, equality follows by taking $h=2^{d/2}\1_{\{x=x_0\}}$ for any cube vertex $x_0$. These bounds reflect the finite size of the cube: once $k\ge d$, the polynomial space no longer grows, so the threshold cannot continue to scale as $e^{2k}$.

\end{document}